\documentclass[11pt,reqno]{amsart}

\usepackage{amsmath, amsfonts, amsthm, amssymb,graphicx, color, enumitem}
\newtheorem{theorem}{Theorem}[section]
\newtheorem{lemma}{Lemma}[section]

\theoremstyle{definition}

\theoremstyle{remark}
\newtheorem{remark}{Remark}[section]

\numberwithin{equation}{section}
\allowdisplaybreaks

\newcommand{\N}{{\mathbb N}}
\newcommand{\R}{{\mathbb R}}

\def\f{\frac}
\renewcommand{\O}{\Omega}

\def\hf1{^\f{1}{1-\xi^2}}
\def\l{\lambda}

\def\be{\begin{equation}}
\def\en{\end{equation}}
\def\bs{\begin{split}}
\def\es{\end{split}}
\def\ba{\begin{align}}
\def\ea{\end{align}}

\def\da{\partial^\alpha}
\def\db{\partial^\beta}
\def\dg{\partial^\gamma}

\def\duno{\partial_1}
\def\ddue{\partial_2}
\def\dtre{\partial_3}

\def\dj{\partial_j}

\def\dt{\partial_t}
\def\div{\rm div\, }

\def\ds{\partial_\star}
\def\dsharp{\partial_\sharp}
\def\du{\partial_1}
\def\dt{\partial_t}
\def\dd{\partial_2}

\def\dj{\partial_j}
\def\dtr{\partial_3}

\def\dnu{\partial_\nu}

\def\F{\mathcal F }
\def\G{\mathcal G }

\renewcommand{\a}{\alpha}

\renewcommand{\b}{\beta}

\def\hf{\hat F}
\def\hm{H^m(  \O)}
\def\hma{H^m_\ast(  \O)}
\def\hmua{H^{m-1}_\ast(  \O)}
\def\hmaa{H^m_{\ast\ast}(  \O)}

\def\hmaaa{H^m_{\ast\ast\ast}(  \O)}

\def\a{\alpha}
\def\aa{{\ast\ast}}
\def\aaa{{\ast\ast\ast}}

\def\b{\beta}

\def\cC
{{\mathcal C}}

\def\cL
{{\mathcal L}}

\def\ctm{\cC_T(H^m)}
\def\ctma{\cC_T(H^m_\ast)}
\def\ctmua{\cC_T(H^{m-1}_\ast)}
\def\ctmda{\cC_T(H^{m-2}_\ast)}
\def\ctmta{\cC_T(H^{m-3}_{\ast})}
\def\ctmaa{\cC_T(H^m_{\ast\ast})}
\def\ctmuaa{\cC_T(H^{m-1}_{\ast\ast})}
\def\ctmdaa{\cC_T(H^{m-2}_{\ast\ast})}
\def\ctmtaa{\cC_T(H^{m-3}_{\ast\ast})}
\def\ctra{\cC_T(H^r_\ast)}
\def\ctraa{\cC_T(H^r_{\ast\ast})}

\def\inn{\quad\hbox{in }}

\author[P. Secchi]{Paolo Secchi}
\address{INdAM Unit \& Department of Civil, Environmental, Architectural Engineering and Mathematics (DICATAM), University of Brescia, Via Valotti 9, 25133 Brescia, Italy}
\email{paolo.secchi@unibs.it}

\title[Incompressible Limit in MHD]
{The Incompressible Limit of the Equations of Compressible Ideal Magneto-Hydrodynamics with perfectly conducting boundary}


\begin{document}

\begin{abstract}
We consider the initial-boundary value problem in the halfspace for the system of equations
of ideal
Magneto-Hydrodynamics with a perfectly conducting wall boundary condition. 
We show the
convergence of solutions
to the solution of the equations of incompressible MHD as the Mach number goes
to zero. Because of the characteristic boundary, where a loss of regularity in the
normal direction to the boundary may occur, the convergence is shown
in suitable anisotropic Sobolev spaces which take account of the
singular behavior at the boundary.
\end{abstract}

\date{\today}

\maketitle



\section{Introduction}
\label{sect1}

We consider the equations of ideal Magneto-Hydrodynamics (MHD) for the
motion of an electrically conducting fluid, where "ideal" means that
the effect of viscosity and electrical resistivity is neglected (see \cite{freidberg}):
\begin{subequations}\label{mhd}
\begin{align}
\rho_p(p^\l)(\partial_t+v^\l\cdot\nabla)p^\l+\rho^\l\nabla\cdot v^\l=0,\label{mhd1}\\
\rho^\l(\partial_t+(v^\l\cdot\nabla))v^\l+\l^2\nabla p^\l+\mu
H^\l\times(\nabla\times
H^\l)=0,\label{mhd2}\\
(\partial_t+(v^\l\cdot\nabla))H^\l-(H^\l\cdot\nabla)v^\l + H^\l\nabla\cdot
v^\l=0.\label{mhd3}
\end{align}
\end{subequations}
Here the pressure $p^\l=p^\l(t,x)$, the velocity field $v^\l=v^\l(t,x)=
(v^\l_1,v^\l_2,v^\l_3)$, the magnetic
field $H^\l=H^\l(t,x)=(H^\l_1,H^\l_2,H^\l_3)$ are unknown functions of
time $t$ and space variables $x=(x_1,x_2,x_3)$. The density $\rho^\l$ is
given by the
equation of state $\rho^\l=\rho(p^\l)$ where
$\rho>0$ and $\partial \rho/\partial p\equiv\rho_p>0$ for $p>0$. The
magnetic permeability $\mu$
is set equal to $1$ without loss of generality. The coefficient $\l$ is
essentially the
inverse of the Mach number. We denote 
$\partial_t=\partial/\partial t,\partial_i=\partial/\partial
x_i,\nabla=(\partial_1,\partial_2,\partial_3)$ and use the conventional
notations of vector
analysis. The system \eqref{mhd} is supplemented with the divergence constraint
\begin{equation}\label{divH}
\nabla\cdot H^\l=0
\end{equation}
on the initial data.

We
study the initial-boundary value problem corresponding to a perfectly
conducting wall
boundary condition. Set $\Omega=\R^3_+=\{x_1>0\}$ and let us denote its
boundary
by $\Gamma$. We also denote $Q_T=(0,T)\times\Omega,\,
\Sigma_T=(0,T)\times\Gamma$ and denote by $\nu=(-1,0,0)$ the unit outward normal to $\Gamma$. We are interested in the study of
the initial-boundary value problem under the boundary conditions
\begin{equation}\label{bc}
v^\l \cdot\nu=0,\quad H^\l\cdot\nu=0 \quad\rm{on}\quad \Sigma_T.
\end{equation}
System \eqref{mhd} -- \eqref{bc} is supplemented with initial conditions
\begin{equation}\label{ic}
(p^\l,v^\l,H^\l)_{|t=0}=(p^\l_0,v^\l_0,H^\l_0) \quad\rm{in}\quad \Omega.
\end{equation}

We study the singular limit as $\l\rightarrow+\infty$.  The limit equations to
system \eqref{mhd} - \eqref{bc}
are
\begin{equation}\label{limit-equ}
\begin{array}{ll}
\overline\rho(\partial_t+(w\cdot\nabla))w+\nabla
\pi+B\times(\nabla\times B)=0,\\
\partial_tB+(w\cdot\nabla)B-(B\cdot\nabla)w=0\\
\nabla\cdot w=0,\qquad \nabla\cdot H=0,\qquad&\hbox{in }Q_T,\\
w\cdot\nu=0,\quad B\cdot\nu=0\quad&\hbox{on }\Sigma_T,\\
(w,B)_{|t=0}=(w_0,B_0)\quad\quad&\hbox{in }\Omega,
\end{array}
\end{equation}
where  $\overline\rho=\rho(0)$ and $w_0$ is such that $\nabla\cdot w_0=0$ in $\Omega$ and
$w_0\cdot\nu=0$ on $\Gamma$ and
analogously for $B_0$.

 \eqref{mhd} - \eqref{bc} is an example of initial boundary value problem for
quasilinear symmetric hyperbolic systems with characteristic boundary. Because of a
possible loss of derivatives in
the normal direction to the boundary, see \cite{MR2289911,MR336093}, in general the solution of such mixed problems
is not in the usual Sobolev space $H^m(\Omega)$, as for the non-characteristic case,
but in the anisotropic weighted Sobolev space $H^m_\ast(\Omega)$ (the definition is
given in the next section). Problem (1.1) - (1.3) was first studied by Yanagisawa and Matsumura \cite{MR1092572}.
As regards the loss of regularity for the solutions of the MHD
equations, Ohno-Shirota \cite{MR1658400}
prove that a mixed problem for the linearized MHD equations is ill-posed in
$H^m(\Omega)$ for $m\geq 2$. A general
regularity
theory for linear and quasilinear systems with characteristic boundary may be found in \cite{MR1346224,MR1401431,S96MR1405665}, see also \cite{MR1070840}. The
application to (1.1) - (1.3) is given in \cite{MR1314493}, where we prove the well-posedness in
the sense of Hadamard in the space $H^m_\ast(\Omega)$. In \cite{MR1941267} we improve the result of
\cite{MR1314493} as we show the solvability in the anisotropic Sobolev space $\hmaa$,
hence obtaining a better regularity than in $\hma$. This also allows to decrease the
smallest order of regularity from $m\geq 8$ to $m\geq 6$.

The initial-boundary value problem for
the incompressible MHD equations (1.4) was studied in \cite{MR1257135}. 

As is well-known from
Fluid Mechanics, one can derive formally the incompressible models such as the
Navier-Stokes equations, the Euler equations or the system \eqref{limit-equ} from the compressible
ones, namely the compressible Navier-Stokes, Euler equations or the equations of ideal
MHD. This corresponds to passing to the limit in the appropriate non dimensional form as
the Mach number goes to zero. This formal
derivation is the argument of many papers in the case of the Euler equations, see \cite{MR1320000,MR615627,MR748308, MR1834114}, the papers \cite{MR925620,MR1458527,MR918838,MR849223} where the phenomenon of the initial layer is
considered; the papers \cite{MR2106119,MR834481,MR1765773} in the presence of the boundary. In the MHD case, the
rigorous derivation of (1.4) from the non dimensional system (1.1) - (1.3) is shown
in \cite{MR665380,MR3452249,MR3962511,MR4665038,MR615627}. See also \cite{MR3803773,MR1039472,MR4028271,MR1339828} for the small Alfv\'en number limit.

In the present paper we show
such derivation for the initial-boundary value problem \eqref{mhd}--\eqref{bc}. While for the periodic boundary
conditions the analysis is essentially a repetition of that for the Euler
equations, in our case we need completely different and much more subtle
arguments. Because of the above mentioned singular behavior at the boundary, we work in
the anisotropic Sobolev space $\hmaa$. 

The most difficult part is to prove the
uniform estimate \eqref{u-uniform} in Theorem \ref{th-uniform}. 
Sections 3 and 4 are dedicated to this proof. 
We
adapt the approach of \cite{MR1941267}, by first proving the apriori estimates of purely tangential and tangential and first
order normal derivatives. Then we show separate apriori estimates of the higher normal derivatives of the non-characteristic part of the solution, that is the part corresponding to the invertible part of the boundary matrix, and the characteristic part of the solution. 
At this point it is crucial to exploit the higher regularity of the non-characteristic part in the function space $\hmaaa$, with respect to the regularity of the characteristic part in the other space $\hmaa$. 
{A particular care is needed for the proof of the equivalence of the weighted normal derivative with a function vanishing at the boundary, in terms of the anisotropic regularity of the function, and the conormal derivative.} The proof of the uniform estimate is totally depending on the general structure \eqref{structure} of the equations, as {highlighted} in \cite{MR748308}, so we believe that the same method could work for other problems with characteristic boundary and a similar structure.

 The convergence is shown in the presence of or without
the initial layer in Theorems \ref{weak-conv} and \ref{strong-conv} respectively, depending whether the limiting
initial velocity $w_0$ satisfies $\div w_0\not=0$ or $\div w_0=0$. For, we use a compactness
argument and, in the first case, also exploit the asymptotic behavior of solutions to
the linearized acoustic equations in the unbounded domain following the argument of \cite{MR1458527}. A similar argument shows
the strong convergence of the gradient of the total pressure, see \eqref{2.19}, Theorem \ref{strong-conv}. The proofs of Theorems \ref{weak-conv} and \ref{strong-conv} are given in Sections 5 and 6, respectively.

Recently there appeared the paper \cite{wang-zhang}, where the authors study the same problem with a different approach.

\section{Formulation of the problem, notations and main results}

Let us introduce the total pressure $q^\l=\l p^\l+(1/2\l)|H^\l|^2$.
In terms of $q^\l$ the equation for the pressure \eqref{mhd1} takes
the form
\begin{equation}\label{mhd21}
\frac{\rho_p^\l}{\rho^\l}\{(\partial_t+v^\l\cdot\nabla)q^\l-\frac{H^\l}{\l}
\cdot(\partial_t+(v^\l\cdot\nabla))H^\l\}+\l\nabla\cdot
v^\l=0,
\end{equation}
where it is understood that
$\rho^\l=\rho(q^\l/\l-(1/2)|H^\l/\l|^2)$ and similarly for
$\rho_p^\l$. Since $H^\l\times(\nabla\times
H^\l)=(1/2)\nabla|H^\l|^2-(H^\l\cdot\nabla)H^\l$, the equation
for the velocity \eqref{mhd2}  takes the form
\begin{equation}\label{mhd22}
\rho^\l(\partial_t+(v^\l\cdot\nabla))v^\l+\l\nabla q^\l-
(H^\l\cdot\nabla)H^\l=0.
\end{equation}
Finally we derive $\nabla\cdot v^\l$ from (2.1) and rewrite the equation
for the magnetic field \eqref{mhd3}  as
\begin{equation}\label{mhd23}
(\partial_t+(v^\l\cdot\nabla))H^\l-(H^\l\cdot\nabla)v^\l -
\frac{H^\l}{\l}\frac{\rho_p^\l}{\rho^\l}\{(\partial_t+v^\l\cdot\nabla)q^\l-\frac{H^\l}{\l}
\cdot(\partial_t+(v^\l\cdot\nabla))H^\l\}=0.
\end{equation}

We rewrite \eqref{mhd21}--\eqref{mhd23}  as
\begin{equation} 
\begin{array}{ll}\label{mhdq}

\begin{pmatrix}
 {\rho_p/\rho}&\underline
0&-({\rho_p/\rho})H^\l/\l\\
\underline 0^T&\rho I_3&0_3\\
-({\rho_p/\rho})(H^\l/\l)^T&0_3&a_0^\l
 \end{pmatrix}
 \dt
 \begin{pmatrix}
 q^\l\\ 
 v^\l\\
H^\l
\end{pmatrix}
+\\
\\
\begin{pmatrix}
(\rho_p/\rho)
v^\l\cdot\nabla&\l\nabla\cdot&-({\rho_p/\rho})(H^\l/\l) v^\l\cdot\nabla\\
\l\nabla&\rho v^\l\cdot\nabla I_3&-H^\l\cdot\nabla I_3\\
-({\rho_p/\rho})(H^\l/\l)^T v^\l\cdot\nabla&-H^\l\cdot\nabla I_3&a_0^\l
v^\l\cdot\nabla
\end{pmatrix}
\begin{pmatrix}
q^\l  \\ v^\l \\ H^\l
\end{pmatrix}=0\,,
\end{array}
\end{equation}
where $\underline 0=(0,0,0)$ and
$$
a_0^\l=I_3+({\rho_p/\rho})(H^\l/\l)\otimes(H^\l/\l).$$It is well known
that the constraint \eqref{divH} can be seen just as a restriction
on the initial value $H^\l_0$. Under this restriction, \eqref{mhdq}  is equivalent to \eqref{mhd}.
The quasilinear symmetric system \eqref{mhdq} is hyperbolic if the state equation $\rho=\rho(p,S)$ satisfies the hyperbolicity condition:
\begin{equation}\label{hyper}
\rho>0,\quad \rho_p>0.
\end{equation}
We write \eqref{mhdq} in compact form as
\begin{equation}\label{compact}
Lu^\l=A_0\partial_tu^\l
+\sum^n_{j=1}(A_j+\l C_j)\partial_ju^\l=0\,,
\end{equation}
where $u^\l=(q^\l,v^\l,H^\l)$ and where the coefficient matrices have the
special structure
(see \cite{MR748308}) 
\begin{subequations}\label{structure}
\begin{align}
& A_j=A_j(u^\l,u^\l/\l)\qquad j=1,2,3,\label{structure1}\\
& A_0=A_0(u^\l/\l)\label{structure2}\\
&  C_j\;\hbox{ are constant symmetric matrices}\label{structure3}\\
&  A_\nu:=\sum_{j=1}^3A_j\nu_j=-A_1 \equiv 0\quad\hbox{on}\,\,\Sigma_T.\label{structure4}
\end{align}
\end{subequations}
Let us define
\[
M=
\begin{pmatrix}
 0&\nu&\underline 0&0\\
  0&\underline 0&\nu&0
\end{pmatrix},
\]
$u_0^\l=(q^\l_0,v^\l_0,H^\l_0)$ where $q^\l_0=\l
p^\l_0+(1/2\l)|H^\l_0|^2$.
We rewrite (1.1), (1.2), (1.3)
as
\begin{equation}
\begin{array}{ll}\label{ibvp}
Lu^\l=0&\quad\hbox{on}\,\, Q_T,\\
Mu^\l=0&\quad\hbox{on}\,\,\Sigma_T,\\
u^\l_{|t=0}=u_0^\l&\quad\hbox{on}\,\,\Omega.
\end{array}
\end{equation}
When, as in
\eqref{ibvp}, we multiply matrices by vectors, vectors have always to be
considered as column
vectors.

Let us introduce some notations. Let $H^m(\Omega)$ be the usual Sobolev
space of order $m,\, m=1,2,...$, and let $||\cdot||_m$ denotes
its norm. The norm of $L^2(\Omega)$ is denoted by $||\cdot ||$, the norm of
$L^p(\Omega),\, 1\leq
p\leq \infty$, by $|\cdot|_p$. Let $\sigma\in C^\infty(\overline{R}_+)$ be a
monotone
increasing function such that $\sigma(x_1)=x_1$ in a neighborhood of the origin and
$\sigma(x_1)=1$ for any $x_1$ large enough. Let us introduce the differential
operator in
the tangential directions (conormal derivative)
$$
\partial^{\alpha}_\ast
=(\sigma(x_1)\partial_1)^{\alpha_1}\partial_2^{\alpha_2}\partial_3^{\alpha_3}
$$
where $\alpha =(\alpha
_1,\a_2,\alpha _3), |\alpha |=\alpha_1+\a_2+\alpha_3$.
Given $m\geq 1$, the function space
$H^m_\ast(\Omega)$ is defined as the set of functions $u\in L^2(\Omega)$
such that
$\partial^{\alpha}_\ast\duno^k u\in
L^2(\Omega)$ if $|\a|+2k\leq m$. Derivatives are considered in the
distribution sense. $H^m_\ast(\Omega)$ is normed by
$$||u||^2_{m,\ast
}=\sum
_{|\alpha|+2k\leq m}||\partial^{\alpha}_\ast
\partial_1^ku||^2.$$
In the sequel we will refer to $|\a|+2k$ as to the order of the operator
$\partial^{\alpha}_\ast\duno^k$. Thus the normal derivative behaves in
$H^m_\ast$ like a
differential operator of order two.
We also define other
function spaces. 

The space $H^m_{\ast\ast}(\Omega),\, m\geq 1,$ consists of
the functions
$u\in L^2(\Omega)$ such that
$\partial^{\alpha}_\ast\duno^k u\in
L^2(\Omega)$ if $|\a|+2k\leq m+1,|\a|\leq m$. $H^m_{\ast\ast}(\Omega)$ is
normed by
$$||u||^2_{m,\ast\ast
}=\sum_{\a,k}
||\partial^{\alpha}_\ast
\partial_1^ku||^2$$
where the sum is taken over all multi-indices $\alpha$ and indices $k$ such
that $|\alpha
|+2k\leq m+1,|\alpha|\leq m$.

The space $\hmaaa,\, m\geq 1,$ consists of the
functions $u\in L^2(\Omega)$ such that
$\partial^{\alpha}_\ast\duno^k u\in
L^2(\Omega)$ if
$|\a|\leq m$ and if $|\a|+2k\leq m+2$ for $k\geq 2$, $|\a|+2k\leq m+1$ for
$k=1$.
$H^m_\aaa(\Omega)$ is normed by  $$||u||^2_{m,\ast\ast\ast
}=\sum_{\a,k}
||\partial^{\alpha}_\ast
\partial_1^ku||^2$$
where the sum is taken over all multi-indices $\alpha$ and indices $k$ such
that $|\alpha
|+2k\leq m+2$ if $k\geq2$, $|\alpha
|+2k\leq m+1$ if $k=1$, and such that $|\alpha|\leq m$.
Clearly
$$
H^m(\Omega)\hookrightarrow \hmaaa\hookrightarrow
H^m_{\ast\ast}(\Omega)\hookrightarrow H^m_{\ast}(\Omega)\subset H^m_{loc}(\Omega),
$$
$$
H^m_{\ast}(\Omega)\hookrightarrow H^{[{m\over 2}]}(\Omega),\qquad
H^m_{\ast\ast}(\Omega)\hookrightarrow H^{[{m+1\over 2}]}(\Omega),
$$
where $[{m\over 2}]$ and $[{m+1\over 2}]$ denote the integer part of ${m\over 2}$ and ${m+1\over 2}$, respectively.
In particular $H^1_{\ast\ast}(\Omega)=H^1(\Omega)$. For the
sake of convenience we set
$H^0_{\ast\ast\ast}(\Omega)=H^0_{\ast\ast}(\Omega)=H^0_{\ast}(\Omega)=L^2(\Omega
)$. Let $X$
be a Banach space and let $T>0$; then $C([0,T];X)$, $L^{\infty}(0,T;X)$
denote respectively the space of continuous and essentially bounded functions defined
on $[0,T]$ taking values in $X$. $C^k([0,T];X)$ denotes the space of
$k$-continuously differentiable
functions on $[0,T]$ with values in $X$; $W^{k,\infty}(0,T;X)$ is the
space of essentially bounded functions
together with the
derivatives up to order $k$, defined on $[0,T]$ taking values in $X$. We define

\[
{\mathcal C}_T(H^m_\ast)=\bigcap_{k=0}^mC^k([0,T];H^{m-k}_\ast(\Omega)),\qquad
{\mathcal L}^{\infty}_T(H^m_\ast)= \bigcap_{k=0}^m
W^{k,\infty}(0,T;H^{m-k}_\ast(\Omega));\]
the norm is ($ess\,sup$ in the second case)
$$|||u|||_{m,\ast ,T}=\sup_{[0,T]}|||u(t)|||_{m,\ast},$$
where
$$
|||u(t)|||^2_{m,\ast }=\sum
_{k=0}^m||\partial^k_tu(t)||^2_{m-k,\ast }=\sum
_{|\gamma|+2h\leq m}||\partial_\star^\gamma\partial^h_1u(t)||^2,
$$
for
$\gamma=(k,\alpha),|\gamma|=k+|\alpha|,\partial^\gamma_\star=
\partial_t^k\partial^\alpha_\ast$.
We also define 
\[
{\mathcal C}_T(H^m_\aa)=\bigcap_{k=0}^mC^k([0,T];H^{m-k}_\aa(\Omega)),\qquad
{\mathcal L}^{\infty}_T(H^m_\aa)= \bigcap_{k=0}^m
W^{k,\infty}(0,T;H^{m-k}_\aa(\Omega)),
\]
with the norm ($ess\,sup$ in the second case)
$$|||u|||_{m,\aa ,T}=\sup_{[0,T]}|||u(t)|||_{m,\aa},$$
where
$$
|||u(t)|||^2_{m,\aa }=\sum
_{k=0}^m||\partial^k_tu(t)||^2_{m-k,\aa }=\sum_{\gamma,h}
||\partial_\star^\gamma\partial^h_1u(t)||^2,
$$
the last sum being taken over all the indices $\gamma=(k,\alpha),|\gamma|=k+|\alpha|,\partial^\gamma_\star=
\partial_t^k\partial^\alpha_\ast$, and $h$ such that $|\gamma |+2h\leq
m+1,|\gamma|\leq m$. Moreover we define the seminorm
$$
||[u(t)]||^2_{m,\aa }=\sum
_{k=0}^m|[\partial^k_tu(t)]|^2_{m-k,\aa }=\sum_{\gamma,h}
||\partial_\star^\gamma\partial^h_1u(t)||^2,
$$
the last sum being taken over all the indices $\gamma=(k,\alpha),|\gamma|=k+|\alpha|,\partial^\gamma_\star=
\partial_t^k\partial^\alpha_\ast$, and $h$ such that $1\leq|\gamma |+2h\leq
m+1,|\gamma|\leq m$.
Similarly we define
${\mathcal
C}_T(H^m),{\mathcal L}^{\infty}_T(H^m)$ by using
$H^{m-k}(\Omega)$ instead
of $H^{m-k}_\aa(\Omega)$. The norm is denoted by
$|||\cdot|||_{m,T}$. In the present context of the study of the singular limit as
$\l\rightarrow\infty$, it is convenient to introduce in $\ctmaa$ the following weighted
norm
$$|||u|||_{m,\aa,\l ,T}=\sup_{[0,T]}|||u(t)|||_{m,\aa,\l},$$
where
$$
|||u(t)|||^2_{m,\aa,\l }=\sum
_{k=0}^m||\l^{-k}\partial^k_tu(t)||^2_{m-k,\aa }.
$$
For each {\it fixed} $\l$, $|||u(t)|||_{m,\aa,\l }$ is equivalent to
$|||u(t)|||_{m,\aa}$. Moreover we define the corresponding seminorm
$$
||[u(t)]||^2_{m,\aa,\l }=\sum
_{k=0}^m|[\l^{-k}\partial^k_tu(t)]|^2_{m-k,\aa }=\sum_{\gamma,h}
||\l^{-k}\partial_\star^\gamma\partial^h_1u(t)||^2,
$$
the last sum being taken over all the indices $\gamma=(k,\alpha),|\gamma|=k+|\alpha|,\partial^\gamma_\star=
\partial_t^k\partial^\alpha_\ast$, and $h$ such that $1\leq|\gamma |+2h\leq
m+1,|\gamma|\leq m$.

Given the system $(2.7)$ for $u^\l$, we recursively define
$u^{\l(k)}_0, k\geq 1$, by formally
taking $k-1$ time derivatives of the equations, solving for
$\partial_t^ku^\l$ and
evaluating it at time $t=0$; for $k=0,u^{\l(0)}_0=u_0^\l$. We set $$
|||u_0^\l|||^2_{m,\aa}=\sum^m_{k=0}||u^{\l(k)}_0||^2_{m-k,\aa}, \qquad
|||u_0^\l|||^2_{m,\aa,\l}=\sum^m_{k=0}||\l^{-k}u^{\l(k)}_0||^2_{m-k,\aa}.
$$

We use the same notations for spaces of
scalar,
vector valued or matrix valued functions. 
Throughout the paper we will denote by $C$
generic constants which may vary from line to line or even in the same
line. Other constants are
denoted by $K_i$, suitable increasing functions are denoted by $\Phi,
\Phi_i$.

\bigskip

The following theorem gives the existence of solutions of (2.7) in $\hmaa$ for {\it
fixed} $\l$. From
it one can easily obtain an existence theorem for the original system (1.1)
- (1.3) with
$p^\l$ instead of $q^\l$. The proof of the theorem may be found in \cite{MR1941267}; the result improves the existence
result in $\hma$ of Yanagisawa-Matsumura \cite{MR1092572} and the author \cite{MR1314493}.

\begin{theorem}[\cite{MR1941267} Existence]\label{th-existence}
 Let $m\geq 6$ be an integer and $\l>0$. Let $\rho\in C^{m+1}$ and
$u_0^\l=(q^\l_0,v^\l_0,H^\l_0)\in  \hm$,
$p^\l_0=q^\l_0/\l-(1/2)|H_0^\l/\l|^2$ be such
that $\rho (p^\l_0)>0,\rho_p(p^\l_0)>0$ in
$\overline{\Omega}$, $\nabla\cdot H^\l_0=0$ in $\Omega$.
The data satisfy the compatibility conditions
$v_0^{\l(k)}\cdot\nu=0,k=0,\dots,m-1,H^\l_0\cdot\nu=0$ on $\Gamma$.

Then there exists $T_\l>0$ such that the mixed problem \eqref{ibvp}
has a unique solution $u^\l=(q^\l,v^\l,H^\l)\in{\mathcal C}_{T_\l}(H^m_\aa)$ with $\rho
(p^\l)>0,\rho_p (p^\l)>0$ for $p^\l=q^\l/\l-(1/2)|H^\l/\l|^2$, $\nabla\cdot H^\l=0$ in
$ Q_{T_\l}$. 
Moreover
$\nu\cdot\partial_t^kv^\l(t)_{|\Gamma}\in H^{m-k-1/2}(\Gamma),k=0,\dots,$
$m-1$,
$\nu\cdot H^\l(t)_{|\Gamma}\in H^{m-1/2}(\Gamma)$, for each $t\in[0,{T_\l}]$. 
\end{theorem}

The following theorems give the results about the singular limit
$\l\rightarrow\infty$.

\begin{theorem}[Uniform boundedness]\label{th-uniform}
Let $m\geq 6$ be an
integer and let the assumptions of Theorem \ref{th-existence} be satisfied.
Assume that
\begin{equation}\label{u0-uniform}
\begin{array}{ll}

||u^\lambda_0||_m\leq
K_1\qquad\forall\lambda\geq 1,
\end{array}
\end{equation}
where $K_1$ is independent of $\l$. Then for each $\l\ge1$ there exists a unique solution
$u^\l=(q^\lambda,v^\lambda,H^\lambda)\in\cC_T(H^m_\aa)$ of \eqref{ibvp}  in
$ Q_{T}$, where $T$ is independent of $\l$. The solutions satisfy the uniform estimate
\begin{equation}\label{u-uniform}
\begin{array}{ll}

|||u^\lambda|||_{m,\ast\ast,\l,T}+|||\dt H^\lambda|||_{m-1,\ast\ast,\l,T}\leq
K_2\qquad\forall\lambda\geq 1,
\end{array}
\end{equation}
with $K_2$ independent of $\l$.
\end{theorem}

If the divergence $\nabla\cdot  w_0\not= 0$, then the initial layer appears as for the
Euler equations \cite{MR925620,MR1458527,MR918838,MR849223}.

\begin{theorem}[Weak convergence]\label{weak-conv}

Let the
assumptions of Theorem 2.2 hold and let $(w_0,B_0)\in H^m(  \O)$ where
$\nabla\cdot B_0=0$ in $\Omega$, $B_0\cdot\nu=0$ on $\Gamma$.
Assume also
that 
\begin{equation}\label{conv-dati}
\begin{array}{ll}
\lim_{\l\rightarrow+\infty}(||v^\lambda_0-w_0||_{m,\ast\ast}
+||H^\lambda_0-B_0||_{m,\ast\ast})=0.
\end{array}
\end{equation}
Then 
\begin{subequations}\label{conv}
\begin{align}
u^\l=(q^\l,v^\l,H^\l)\rightarrow(0,w,B)\quad\text{weakly}-\ast\,\text{ in } L^\infty(0,T;\hmaa),\label{conv-u}
\\
v^\l\rightarrow w\quad\text{strongly in } C_{loc}((0,T]\times\overline  \O),\label{conv-v}
\\
H^\l\rightarrow B\quad\text{in } C([0,T];H^{m-1}_{\ast,loc}(  \O)),\label{conv-H}
\\
\rho^\l\rightarrow \overline\rho
\quad\text{in }C([0,T];H^{m}_{\aa}(  \O)),\label{conv-rho}
\end{align}
\end{subequations}
as
$\l\rightarrow+\infty$, where $(w,B)$ is the unique solution of \eqref{limit-equ} with initial conditions
$$
w_{|t=0}=P_Sw_0, \qquad B_{|t=0}=B_0,
$$where $P_S$ is the projection onto the solenoidal subspace incident to the Helmholtz
decomposition. 
\end{theorem}

\begin{remark}
 (i) In \eqref{conv-v}, $(0,T]$ cannot be
replaced by $[0,T]$ because in general $w$ satisfies $w_{|t=0}=P_Sw_0$ instead of
$w_{|t=0}=w_0$, so that the initial layer develops since the convergence is not uniform
near $t=0$. 

(ii) The convergence in $C([0,T];H^{m-1}_{\ast,loc}(  \O))$ is the convergence
in $C([0,T];H^{m-1}_{\ast}(\O\cap B(0,r)))$ for every $r>0$; a Sobolev imbedding shows
that it implies the strong convergence in $C_{loc}([0,T]\times\overline  \O)$.
\end{remark}

The next theorem gives the strong convergence for well-prepared initial data.
\begin{theorem}[Strong convergence]\label{strong-conv}
Assume the hypotheses of
Theorems \ref{th-uniform} and \ref{weak-conv}. Let $w_0$ be such that
$\nabla\cdot
w_0=0$ in $\Omega$, $w_0\cdot\nu=0$ on $\Gamma$ and let
$(w,B)\in\cC_T(H^m)$ be the corresponding solution to \eqref{limit-equ}. Assume also that 
\begin{equation}\label{well-prep}
\begin{array}{ll}
\lambda||\nabla
q^\lambda_0||_{m-1}+\lambda||\nabla\cdot v^\lambda_0||_{m-1}\leq
K'_1\qquad\forall\lambda\geq 1,

\end{array}
\end{equation}
where $K'_1$ is independent of $\l$. Then
\begin{equation}
\begin{array}{ll}\label{u-uniform2}
|||u^\lambda|||_{m,\ast\ast,\l,T}+|||\dt u^\lambda|||_{m-1,\ast\ast,\l,T}+\l|||\nabla
q^\l|||_{m-1,\ast\ast,\l,T}
+\l|||\nabla\cdot v^\l|||_{m-1,\ast\ast,\l,T}\leq
K'_2
\end{array}
\end{equation}
for all $\lambda\geq 1$, where $K'_2$ and $T$ are independent of $\l$.
Moreover as
$\l\rightarrow+\infty$
\begin{subequations}\label{2.19}
\begin{align}
q^\l\rightarrow 0,\quad v^\l\rightarrow w\quad&\hbox{in }
C([0,T];H^{m-1}_{\ast,loc}(  \O)),\label{conv-qv}\\
 \l\nabla q^\l=\nabla(\l^2p^\l+|H^\l|^2/2)\rightarrow\nabla(\pi+|B|^2/2)\quad&\hbox{in }
C_{loc}((0,T]\times\overline  \O).\label{conv-grad-q}
\end{align}
\end{subequations}

\end{theorem}

\begin{remark}
 (i) The existence of the solution $(w,B)$ to \eqref{limit-equ} in the class
$\ctm$ is shown in \cite{MR1257135}; the total pressure $\pi$ in \eqref{limit-equ} is defined up to a constant
and such that $\nabla\pi\in \cC_T(H^{m-1})$.

(ii) From \eqref{mhd22}, \eqref{conv-dati}--\eqref{well-prep} we know that $(\rho^\l\dt v^\l+\l^2\nabla
p^\l)_{|t=0}\rightarrow(\overline\rho\dt w+\nabla\pi)_{|t=0}$ and that $\l^2\nabla
p^\l_0$ is uniformly bounded; however we don't know whether this last term converges to
$\nabla\pi_{|t=0}$. Therefore one can't obtain in \eqref{conv-grad-q} a uniform convergence near
$t=0$. If $m\geq 6$, from \eqref{conv-H} we can obtain
$\nabla|H^\l|^2\rightarrow\nabla|B|^2$ in $C_{loc}([0,T]\times\overline  \O)$,
so that we can improve \eqref{conv-grad-q} to $\l^2\nabla
p^\l\rightarrow\nabla\pi$ in $C_{loc}((0,T]\times\overline  \O)$.
\end{remark}

\section{Proof of a preliminary a priori estimate}

In this section we drop for convenience the index $\l$. Suppose that the assumptions
of Theorem 2.2 hold. Let $a_0$, depending increasingly on $\omega^{-1}$, be such that
$0<a_0<1$ and $$ 2a_0I_7\leq A_0(u_0/\l)\leq (2a_0)^{-1}I_7\quad\inn\O,\forall\l\geq1.$$
Moreover, let $R>0$ be such that
$$
|\rho(p_0)|+|(\rho_p/\rho)(p_0)|\leq R/2\quad\inn\O,\forall\l\geq1.
$$
Let $u$ be a sufficiently smooth solution of
(1.1)-(1.3), without giving for the moment the precise definition; it is
understood that the following computation for deriving the apriori
estimate can be carried out legitimately. Given $T>0$, we assume that $u$ satisfies
\begin{equation}\label{set}
\begin{array}{ll}
a_0I_7\leq A_0(u/\l)\leq
a_0^{-1}I_7,\quad|\rho|+|\rho_p/\rho|\leq R,\\
\\
|||u|||_{5,\ast\ast,\l}\leq K\quad\inn
Q_T,\;\forall\l\geq 1.
\end{array}
\end{equation}

\medskip
\begin{lemma}\label{stimeA}
Let $u$ be a sufficiently smooth solution of (2.7)
satisfying
\eqref{set}. The matrices $A_0=A_0(u/\l)$, $A=(A_1,A_2,A_3)=A(u,u/\l)$ of (2.4), (2.5)
satisfy 
\begin{subequations}\label{stimeAi}
\begin{align}
||[A_0]||_{m,\ast\ast,\l}&\leq\l^{-1}
\Phi(R,K)|||u|||_{m,\ast\ast,\l},\label{stimaA0}\\
|||A_j|||_{m,\ast\ast,\l}&\leq
\Phi(R,K)|||u|||_{m,\ast\ast,\l}\qquad j=1,2,3,\label{stimaAj}
\\
|Div \vec A|_\infty&\leq \Phi(R,K),\label{stimaDiv}
\end{align}
\end{subequations}
for all $t\in[0,T],\,\l\geq
1$, where $Div \vec A=\dt A_0+\sum_j\dj A_j$ and where $\Phi$ is a suitable increasing
function independent of $\l$.
\end{lemma}

\begin{proof}
By the change of scale $\tau=\l t$,
$||[A_0(u(t)/\l)]||_{m,\aa,\l}=||[A_0(u(\tau/\l)/\l)]||_{m,\aa}$,
$|||u(t)|||_{m,\aa,\l}=|||u(\tau/\l)|||_{m,\aa}$. Lemma \ref{nuovo3} and Theorem \ref{imbedding} yield (here $\ds$ 
includes $\partial/\partial\tau$) 
\begin{equation*}\label{}
\begin{array}{ll}\label{}
||[A_0(u(t)/\l)]||_{m,\aa,\l}=||[A_0(u(\tau/\l)/\l)]||_{m,\aa}\\
\\
\leq |||\ds
A_0((u(\tau/\l)/\l)|||_{m-1,\aa}+|||\duno A_0(u(\tau/\l)/\l)|||_{m-1,\ast}\\
\\
\leq \Phi(R,K)|||u(\tau/\l)/\l|||_{m,\aa}
=\Phi(R,K)|||u(t)/\l|||_{m,\aa,\l}\\
\\
\leq \l^{-1}\Phi(R,K)|||u(t)|||_{m,\aa,\l}
\end{array}
\end{equation*}
for all $\l\geq 1$. \eqref{stimaAj} follows by a similar computation and Lemma A.3. Moreover we
have $ Div \vec A={\partial A_0\over\partial u}\l^{-1}\dt u+\sum_j({\partial
A_j\over\partial u}+\l^{-1}{\partial
A_j\over\partial u/\l})\dj u$. From Lemma \ref{nuovo3} we obtain \eqref{stimaDiv} since $\l^{-1}\dt u,\ddue
u,\dtre u\in \ctmuaa$ and $\duno u\in \ctmdaa\hookrightarrow C^0_B(\overline  \O)$.
\end{proof}

\medskip

Observe that the differentiation of $A_0$ gives one $\l^{-1}$ that can be
associated to $\dt u$ and used as the weight for $\dt$ in the $H^m_\aa$ norms. This fact
shows the importance of the dependence of $A_0$ only on $u/\l$, as pointed out in \cite{MR748308}.
We will denote
\begin{equation}\label{def-norme-A}
\begin{array}{ll}
|||A|||_{m,\ast\ast,\l}=||[A_0(u(t)/\l)]||_{m,\aa,\l}+\sum_j |||A_j|||_{m,\ast\ast,\l}.
\end{array}
\end{equation}

Let $U$ be a sufficiently smooth solution of 
\begin{equation}\label{linear-equ}
\begin{array}{ll}
A_0\partial_tU
+\sum^n_{j=1}(A_j+\l C_j)\partial_jU=F
,
\end{array}
\end{equation}
where the matrix coefficients are evaluated at $u$
and $F$ is a given vector field. For convenience let us set
$\int=\int_\O dx$.

\begin{lemma}\label{lemma-stime-lin}
 Let the solution $U=(Q,V,W)$ of \eqref{linear-equ} (with $Q\in\R$ and $V,W\in\R^2$) be such that
\begin{equation}\label{linear-bc}
\begin{array}{ll}
\hbox{either }Q=0\quad\hbox{or
}V_1=0\quad\hbox{on}\quad\Sigma_T.  
\end{array}
\end{equation}
 Then 
\begin{equation}\label{stima-linear}
\begin{array}{ll}
\displaystyle a_0||U(t)||^2\leq
a_0^{-1}||U(0)||^2+\int^t_0(2\int F\cdot U+|Div \vec A|_\infty||U||^2)ds,  
\end{array}
\end{equation}
for all $t\in[0,T],\,\l\geq
1$. 
\end{lemma}
\begin{proof}
We multiply \eqref{linear-equ} by $U$ and integrate over $\Omega$. An
integration
by parts and standard calculations give
\begin{equation}\label{balance}
\begin{array}{ll}
{d\over dt}\int A_0U\cdot U=\int_{\partial\Omega}(A_1+\l C_1)U\cdot
U+\int(2F+Div \vec A\,U)\cdot
U
\end{array}
\end{equation}
Here we use $\dj(\l C_j)=0$. In our case
$$
(A_1+\l C_1)U\cdot U=2\l QV_1\quad\hbox{on}\quad\Sigma_T.$$
Thus, under \eqref{linear-bc}, the boundary integral in \eqref{balance} vanishes and \eqref{stima-linear} readily
follows. 
\end{proof}

\medskip
We use Lemma \ref{lemma-stime-lin} to get the apriori estimate in $\hmaa$.
\medskip

\begin{lemma}\label{lemma-stime-tang}
Let $m\ge5$. Let $u$ be a sufficiently smooth solution of (2.7)
satisfying
(3.1) and $\a=(\a_0,\dots,\a_3)$ a
multi-index with
$|\a|\leq m$ ($\a_0$ stands for the index of $\dt$). Then $\partial^\alpha_\star u$ satisfies
\begin{equation}\label{}
\begin{array}{ll}\label{stima-tang}
a_0||\l^{-\a_0}\partial^\alpha_\star u(t)||^2\leq
a_0^{-1}||\l^{-\a_0}\partial^\alpha_\star u(0)||^2+\\
\displaystyle +C\int^t_0 \{\Phi(K) ||\l^{-\a_0}\partial^\alpha_\star
u||^2+|||A|||_{m,\ast\ast,\l}|||u|||^2_{m,\ast\ast,\l}+|||u|||_{m,\ast\ast,\l}|||\l\du
q|||_{m-1,\ast,\l}\}ds,
\end{array}
\end{equation}
for all $t\in[0,T],\,\l\geq
1$, where the increasing function $\Phi$ is independent of
$\l$.
\end{lemma}
\begin{proof}
We apply $\l^{-\a_0}\partial^\alpha_\star$ to \eqref{compact} and obtain \eqref{linear-equ} with
$U=U_\a=\l^{-\a_0}\partial^\alpha_\star u$ and
\begin{equation}\label{commut-tang}
\begin{array}{ll}

F=F_\a=-\l^{-\a_0}[\partial^\alpha_\star,A_0\dt]u-\l^{-\a_0}[\partial^\alpha_\star,A_j\dj]u-\l[\da_\star,C_1\du]u.

\end{array}
\end{equation}
Concerning the last term in \eqref{commut-tang} we observe that this term doesn't vanish
even if the
$C_j$ are constant because $[\sigma(x_1)\du,\du]\not=0$. On the contrary we have
$[\partial^\alpha_\star,C_j\dj]=0$ if $j\not=1$.
We observe that, by tangential differentiation along $\Gamma$ of the boundary
condition $v_1=0$,
we can obtain
$
\partial^\alpha_\star v_1=0$ on $\Sigma_T$. Thus we can obtain the estimate for $\l^{-\a_0}\partial^\alpha_\star u$
from Lemma
\ref{lemma-stime-lin}. It is enough to estimate $F_\a$ given by \eqref{commut-tang}. 

Expanding the first commutator we can write
\begin{equation}\label{exp-comm-A0}
\begin{array}{ll}

[\partial^\alpha_\star,A_0\dt]u =\ds^{\a-1}(\ds A_0\dt u)+\ds^{\a-2}(\ds A_0\ds\dt u)+\dots +\ds A_0\ds^{\a-1}\dt u,
\end{array}
\end{equation}
where $\ds^{\a-1}$ means $\ds^\b$ for some multi-index $\b$ with
$\b_i\leq\a_i$ and
$|\b|=|\a|-1$, and so on. Since
$u\in\ctmaa$ and
$\ds A_0\in\ctmuaa$ then, from Lemma \ref{prodotto2} (ii), with $n=3$, we have $\ds A_0\dt
u\in\ctmuaa$, $\ds A_0\ds\dt
u\in\ctmdaa$, $\ds A_0\ds^2\dt u\in\ctmtaa$, and so on (with corresponding multiplicative
inequalities), provided that $m-1\ge4$. Again the differentiation of $A_0$ gives one more $\l^{-1}$ which is then
associated to $\dt u$. Then from \eqref{prodottogen2}, we get
\begin{equation}\label{commutA0}
\begin{array}{ll}

||\l^{-\a_0}[\partial^\alpha_\star,A_0\dt]u||\leq C||
[A_0]||_{m,\ast\ast,\l}|||u|||_{m,\ast\ast,\l},\qquad
t\in[0,T],
\end{array}
\end{equation}
provided that $m-1\ge4$.
Analogously we treat the commutator terms
when $j=2,3$
and obtain
\begin{equation}\label{commutAj}
\begin{array}{ll}

||\l^{-\a_0}[\partial^\alpha_\star,A_j\dj]u||\leq C|||A_j|||_{m,\ast\ast,\l}|||u|||_{m,\ast\ast,\l},\quad
t\in[0,T],\quad j=2,3,
\end{array}
\end{equation}
again provided $m\ge5$.
In case $j=1$ we have
\begin{equation}\label{exp-commutA1}
\begin{array}{ll}

[\partial^\alpha_\star,A_1\du]u=\ds^{\a-1}(\ds A_1\du u)+  \ds^{\a-2}(\ds A_1\ds\du u)+\dots +\ds A_1\ds^{\a-1}\du u.
\end{array}
\end{equation}
If we argument as above we can't obtain the desired estimate because we
have $\ds
A_1\in\ctmuaa,\du u\in\ctmdaa$ which yield $\ds A_1\du u\in\ctmdaa$ from
Lemma \ref{prodotto2} (ii).
Differently, we consider that $u\in\ctmaa$ gives $\du u\in\ctmua$. Then
Lemma \ref{prodotto2} (i)
yields 
$\ds A_1\du
u\in\ctmua$, $\ds A_1\ds\du
u\in\ctmda$, $\ds A_1\ds^2\du u\in\ctmta$, and so on. 
Then from \eqref{prodottogen3},  we get
\begin{equation}\label{commutA1}
\begin{array}{ll}

||\l^{-\a_0}[\partial^\alpha_\star,A_1\du]u||\leq C|||A_1|||_{m,\ast\ast,\l}|||u|||_{m,\ast\ast,\l},\qquad
t\in[0,T],
\end{array}
\end{equation}
again provided $m\ge5$.

Finally we consider the commutator arising from the large operator $\l
C_j\dj$, which
must be treated more carefully. For, it doesn't work to estimate directly
$||\l[\da_\ast,C_1\du]u||$ but we have to consider the integral below. In
what follows
we use the following simple facts ($\sigma=\sigma(x_1)$):
\begin{equation}\label{commut-sigma}
\begin{array}{ll}

\du(\sigma\du)^k=(\sigma'+\sigma\du)^k\du, \qquad
(\sigma\du)^k(\sigma U)=\sigma(\sigma'+\sigma\du)^kU.
\end{array}
\end{equation}
Set 
\[\dsharp^\a=\l^{-\a_0}\dt^{\a_0}\dd^{\a_2}\dtr^{\a_3}, \qquad w=\dsharp^\a u.
\]
We have
\begin{equation*}\label{}
\begin{array}{ll}
\displaystyle  \l\int[\l^{-\a_0}\da_\star,C_1\du]u\cdot\l^{-\a_0}\da_\star u=\l
\int[(\sigma\du)^{\a_1}\du-\du(\sigma\du)^{\a_1}]C_1w\cdot(\sigma\du)^{\a_1}w\\
\displaystyle =\l
\int[(\sigma\du)^{\a_1}-(\sigma'+\sigma\du)^{\a_1}]C_1\du w\cdot(\sigma\du)^{\a_1}w\\
\displaystyle =-\l \sum_{h=0}^{\a_1-1}{\a_1\choose
h}\int(\sigma')^{\a_1-h}(\sigma\du)^h
C_1\du w\cdot
(\sigma\du)^{\a_1}w\\
\displaystyle =-\sum_{h=0}^{\a_1-1}{\a_1\choose
h}\int(\sigma')^{\a_1-h}(\sigma\du)^h
\dsharp^\a{\l\du v_1\choose\l\du q}\cdot
\dsharp^\a
(\sigma\du)^{\a_1-1}{\sigma\du q\choose\sigma\du v_1}. 
\end{array}
\end{equation*}
Now we use the second relation in \eqref{commut-sigma} in the first row of the last
scalar product and
obtain
\begin{equation}\label{3.13}
\begin{array}{ll}
\displaystyle -\sum_{h=0}^{\a_1-1}{\a_1\choose
h}\int(\sigma')^{\a_1-h}\{(\sigma\du)^{h+1}
\dsharp^\a v_1\cdot
\dsharp^\a
(\sigma'+\sigma\du)^{\a_1-1}(\l\du q)+\cr
\displaystyle  \qquad \qquad +(\sigma\du)^h
\dsharp^\a(\l\du q)\cdot\l^{-\a_0}
\da_\star v_1 \}.

\end{array}
\end{equation}
We estimate \eqref{3.13} and obtain
\begin{equation}\label{3.14}
\begin{array}{ll}

|\l\int[\l^{-\a_0}\da_\star,C_1\du]u\cdot\l^{-\a_0}\da_\star u|\leq C|||v_1|||_{m,\ast,\l}|||\l\du
q|||_{m-1,\ast,\l},\quad
t\in[0,T].
\end{array}
\end{equation}
From Lemma \ref{lemma-stime-lin} applied to $\l^{-\a_0}\partial^\alpha_\star u$ and \eqref{stimeAi}, \eqref{commutA0}--\eqref{commutA1}, \eqref{3.14} we obtain
\eqref{stima-tang}. 
\end{proof}

\bigskip

\begin{lemma}\label{lemma-tang-1-norm}
Let $m\ge6$. Let $u$ be a sufficiently smooth solution of \eqref{ibvp}
satisfying
\eqref{set} and $\b=(\b_0,\dots,\b_3)$ a
multi-index with
$|\b|\leq m-1$. Then $\partial^\b_\star\du u$ satisfies
\begin{equation}\label{tang-1-norm}
\begin{array}{ll}
\displaystyle a_0||\l^{-\b_0}\partial^\b_\star\du u(t)||^2\leq
a_0^{-1}||\l^{-\b_0}\partial^\b_\star\du u(0)||^2+\\
\displaystyle   +C\int^t_0\{\Phi(K) ||\l^{-\b_0}\partial^\b_\star\du
u||^2+|||A|||_{m,\ast\ast,\l}|||u|||_{m,\ast\ast,\l}^2+|||\du v_1|||_{m-1,\ast\ast,\l}|||\l\du
q|||_{m-1,\ast,\l}\}ds,
\end{array}
\end{equation}
for all $t\in[0,T],\,\l\geq
1$, where the increasing function $\Phi$ is independent of $\l$.
\end{lemma}

\begin{proof}
We apply $\l^{-\b_0}\partial^\b_\star\du$ to \eqref{compact} and obtain \eqref{linear-equ} with
$U=U_\b=\l^{-\b_0}\partial^\b_\star\du u$ and
\begin{equation}\label{commut-tang-1-norm}
\begin{array}{ll}

F=F_\b=-\l^{-\b_0}[\partial^\b_\star\du,A_0\dt]u-\l^{-\b_0}[\partial^\b_\star\du,A_j\dj]u-\l[\db_\star\du,C_1\du]u.

\end{array}
\end{equation}
From \eqref{mhd22} we have that
\begin{equation}\label{duno q}
\begin{array}{ll}
\l\du q=-\rho(\dt v_1+v\cdot\nabla v_1)+H\cdot\nabla H_1.
\end{array}
\end{equation}
From the
boundary condition
$v_1=H_1=0$ at $\Sigma_T$, it follows that $\du q=0$ at $\Sigma_T$. By tangential
differentiation along
$\Gamma$, we can obtain
$$
\partial^\b_\star\du q=0\quad\hbox{on}\,\Sigma_T.$$
Thus an $L^2$ estimate of $\l^{-\b_0}\partial^\b_\star\du u$ follows by application of Lemma
\ref{lemma-stime-lin}. Again it
suffices to estimate $F_\b$ given by \eqref{commut-tang-1-norm}. Expanding the first commutator we can write
$$
[\partial^\beta_\star\duno,A_0\dt]u =\ds^\b(\duno A_0\dt u)+\ds^{\b-1}(\ds A_0\duno\dt u)+\dots +\ds A_0\ds^{\b-1}\duno\dt u,
$$
where $\ds^{\b-1}$ means $\ds^\gamma$ for some multi-index $\gamma$ with
$\gamma_i\leq\b_i$ and
$|\gamma|=|\b|-1$, and so on. 
Since 
$\du A_0\in\ctmua$ and $\dt u\in\ctmuaa$, then, from Lemma \ref{prodotto2} (i) with $n=3$, we have $\du A_0\dt
u\in\ctmua$ (with corresponding multiplicative
inequalities), provided $m-1\ge4$. 
Since 
$\ds A_0\in\ctmuaa$ and $\du\dt u\in\ctmda$, then from Lemma \ref{prodotto2} (i) we have $\ds A_0\du\dt
u\in\ctmda$, $\ds A_0\ds\du\dt
u\in\ctmta$, and so on, provided $m-1\ge4$. 
Again the differentiation of $A_0$ gives one more $\l^{-1}$ which is then
associated to $\dt u$. Then from \eqref{prodottogen3} we get
\begin{equation}\label{commutB0}
\begin{array}{ll}
||\l^{-\b_0}[\partial^\b_\star\du,A_0\dt]u||\leq C||
[A_0]||_{m,\ast\ast,\l}|||u|||_{m,\ast\ast,\l},\qquad
t\in[0,T],
\end{array}
\end{equation}
provided $m-1\ge4$.
Analogously we obtain
\begin{equation}\label{commutBj}
\begin{array}{ll}
||\l^{-\b_0}[\partial^\b_\star\du,A_j\dj]u||\leq C||
|A_j|||_{m,\ast\ast,\l}|||u|||_{m,\ast\ast,\l},\quad
t\in[0,T],
\quad
j=2,3.
\end{array}
\end{equation}
Expanding the commutator with $A_1$ we can write
$$
[\partial^\beta_\star\duno,A_1\du]u =\ds^\b(\duno A_1\du u)+\ds^{\b-1}(\ds A_1\duno^2 u)+\ds^{\b-2}(\ds A_1\ds\duno^2 u)+\dots +\ds A_1\ds^{\b-1}\duno^2 u.
$$
Concerning the first term in the right hand side, as $\du A_1\in\ctmua$ and $\du u\in\ctmua$ we obtain from Lemma \ref{prodotto} that $\du A_1\du u\in\ctmua$,
provided $m-1\ge4$.
Concerning the other terms in the expansion, we first consider that
$A_1=0$ on $\Sigma_T$ yields,
by tangential
differentiation, $\ds A_1=0$ on $\Sigma_T$. Then we have $\ds A_1/\sigma\in\ctmda$ from \eqref{hoH2}. Since$(\sigma\du)\du u\in\ctmda$ we obtain $(\ds A_1/\sigma)(\sigma\du)\du u\in\ctmda$, $(\ds A_1/\sigma)\ds(\sigma\du)\du u\in\ctmta$, \dots,  $(\ds A_1/\sigma)\ds^{\b-1}(\sigma\du)\du u\in \cC_T(L^2)$ (with corresponding multiplicative
inequalities), provided that $m-2\ge4$.
Thus it follows that
\begin{equation}\label{commutB1}
\begin{array}{ll}
||\l^{-\b_0}[\partial^\b_\star\du,A_1\du]u||\leq
C ||
|A_1|||_{m,\ast\ast,\l}|||u|||_{m,\ast\ast,\l},\quad
t\in[0,T],
\end{array}
\end{equation}
provided that $m\ge6$.


Finally we consider the commutator arising from the large operator.
Set 
\[\dsharp^\b=\l^{-\b_0}\dt^{\b_0}\dd^{\b_2}\dtr^{\b_3}, \qquad  z=\dsharp^\b\du u.
\]
We have
\begin{equation}\label{3.21}
\begin{array}{ll}
\displaystyle  \l\int[\l^{-\b_0}\db_\star\du,C_1\du]u\cdot\l^{-\b_0}\db_\star\du u=\l
\int[(\sigma\du)^{\b_1}\du-\du(\sigma\du)^{\b_1}]C_1z\cdot(\sigma\du)^{\b_1}z \\
\displaystyle  =\l
\int[(\sigma\du)^{\b_1}-(\sigma'+\sigma\du)^{\b_1}]\du C_1z\cdot(\sigma\du)^{\b_1}z \\
\displaystyle  =-\l \sum_{h=0}^{\b_1-1}{\b_1\choose
h}\int(\sigma')^{\b_1-h}(\sigma\du)^h
\du C_1z\cdot
(\sigma\du)^{\b_1}z 
\\
\displaystyle  =-\sum_{h=0}^{\b_1-1}{\b_1\choose
h}\int(\sigma')^{\b_1-h}(\sigma\du)^h
\dsharp^\b\du{\du v_1\choose\l\du q}\cdot
\l^{-\b_0}\db_\star{\l\du q\choose\du v_1}
\\
\displaystyle  =-\sum_{h=0}^{\b_1-1}{\b_1\choose
h}\int(\sigma')^{\b_1-h}\big\{(\sigma\du)^{h}\dsharp^\b
\du (\du v_1)\cdot
\l^{-\b_0}\db_\star(\l\du q)\\
\displaystyle \qquad\qquad\qquad\qquad\quad +(\sigma\du)^{h+1}
\dsharp^\b(\l\du q)\cdot
\du(\sigma\du)^{\b_1-1}\dsharp^\b (\du v_1) \big\}.

\end{array}
\end{equation}
Now we estimate \eqref{3.21} and obtain
\begin{equation}\label{3.22}
\begin{array}{ll}
|\l\int[\l^{-\b_0}\db_\star\du,C_1\du]u\cdot\l^{-\b_0}\db_\star\du u|\leq
C|||\du v_1|||_{m-1,\ast\ast,\l}|||\l\du
q|||_{m-1,\ast,\l}.
\end{array}
\end{equation}
From Lemma \ref{lemma-stime-lin} applied to $\l^{-\b_0}\db_\ast\du u$ and \eqref{stimeAi}, \eqref{commutB0}--\eqref{commutB1} and \eqref{3.22} we obtain
\eqref{tang-1-norm}.
\end{proof}


\bigskip
We decompose $u$ as $u=(u^{I},u^{I\! I})$ where $u^{I}=(q,v_1)$ and $u^{I\!
I}=(v_2,v_3,H_1,H_2,H_3)$. Accordingly we write the matrix coefficients of
$L$ in block
form as
$$
A_0=
\begin{pmatrix}
A_0^{II}& A_0^{I\,I\! I}\cr  A_0^{I\! I\,I}& A_0^{I\! I\,I\!
I}
\end{pmatrix}
$$
and similarly for $A_j,C_j,j=1,2,3$. Now we estimate $\du u^{I\! I}$. We consider the
rows 3 to 7 of \eqref{mhdq} that we
write as
\begin{equation}\label{3.25}
\begin{array}{ll}
(A_0^{I\! I\,I\! I}\dt+\sum_{j=1}^3A_j^{I\! I\,I\! I}\dj)u^{I\! I}=G
\end{array}
\end{equation}
with
\begin{equation}\label{3.26}
\begin{array}{ll}\displaystyle
G=\begin{pmatrix}
G_1 \\
 G_2
\end{pmatrix},\quad 
G_1=\begin{pmatrix}
 -\l\dd q \\
-\l\dtr q
\end{pmatrix},\quad
G_2=\frac{\rho_p}{\rho}\frac{H}{\l}(\dt q+v\cdot\nabla q)+\begin{pmatrix}
 H\cdot\nabla
v_1 \\
 0\\
 0 
\end{pmatrix}
\end{array}
\end{equation}


\begin{lemma}\label{lemma-tang-k-norm}

Let $m\ge6$. Let $u$ be a sufficiently smooth solution of \eqref{ibvp}
satisfying
\eqref{set}, decomposed as
above
as $u=(u^{I},u^{I\! I})$. Then $u^{I\! I}$ satisfies
\begin{equation}\label{tang-k-norm}
\begin{array}{ll}
\displaystyle  a_0||\l^{-\gamma_0}\dg_\star\du^k u^{I\! I}(t)||^2\leq
a_0^{-1}||\l^{-\gamma_0}\dg_\star\du^k u^{I\! I}(0)||^2
+C\int^t_0 \{\Phi(K)||\l^{-\gamma_0}\dg_\star\du^k u^{I\!
I}||^2\\
\displaystyle  \quad + (1+|||A|||_{m,\ast\ast,\l})\left(|||u|||_{m,\ast\ast,\l}+|||\l\du
q|||_{m-1,\ast\ast,\l}+|||\du
v_1|||_{m-1,\ast\ast,\l}\right)|||u|||_{m,\ast\ast,\l}\}ds,
\end{array}
\end{equation}
for all $t\in[0,T],\,\l\geq
1$, and for any multi-index $\gamma$ and integer $k\geq 2$ such that $|\gamma|+2k\leq
m+1$. The constant $C$ is independent of $\l$. 
\end{lemma}
\begin{proof}
We apply $\l^{-\gamma_0}\dg_\star\du^k$ to (3.25) and obtain
\begin{equation}\label{equ-char}
\begin{array}{ll}
(A_0^{I\! I\,I\! I}\dt+\sum_{j=1}^3A_j^{I\! I\,I\! I}\dj)\l^{-\gamma_0}\dg_\star\du^ku^{I\!
I}=K',
\end{array}
\end{equation}
where
\begin{equation}\label{K}
\begin{array}{ll}
K'=\l^{-\gamma_0}\dg_\star\du^kG
-\l^{-\gamma_0}[\dg_\star\du^k,A_0^{I\! I\,I\! I}\dt]u^{I\! I}-\l^{-\gamma_0}[\dg_\star\du^k,A_j^{I\!
I\,I\! I}\dj]u^{I\!
I}. 
\end{array}
\end{equation}
Set $V=\l^{-\gamma_0}\dg_\star\du^ku^{I\! I}$. We multiply (3.28) by $V$ and integrate
over $\Omega$. An
integration by parts (where we use $A_1^{I\! I\,I\! I}=0$ at $\Gamma$) gives
\begin{equation}\label{balance-char}
\begin{array}{ll}
{d\over dt}\int A_0^{I\! I\,I\! I}V\cdot V=\int(2K'+Div A^{I\! I\,I\! I}V)\cdot
V
\end{array}
\end{equation}
where $Div A^{I\! I\,I\! I}=\dt A_0^{I\! I\,I\! I}+\sum_j\dj A_j^{I\!
I\,I\! I}$. We
have
\begin{equation}\label{div-char}
\begin{array}{ll}
||Div A^{I\! I\,I\! I}||_{L^\infty}\leq C|||A|||_3\leq C|||A|||_{5,\ast\ast}.\end{array}
\end{equation}
We estimate $K'$ in \eqref{K}. We expand the first commutator in the form
\begin{equation}
\begin{array}{ll}\label{exp-commut-A0}
\displaystyle   [\dg_\star\du^k,A_0^{I\! I\,I\! I}\dt]u^{I\! I}=\sum_{h=0}^{k-1}\dg_\star\duno^{k-h-1}(\du A_0^{I\! I\,I\! I}\dt\du^h u^{I\! I})
\\
\displaystyle\qquad\qquad\qquad\quad  +\sum_{h=1}^{k-1}\dg_\star\duno^{k-h}( A_0^{I\! I\,I\! I}\dt\du^hu^{I\! I}) + [\dg_\star,A_0^{I\! I\,I\! I}\dt]\du^ku^{I\! I}.
\end{array}
\end{equation}
For the first term we observe that $\du A_0^{I\! I\,I\! I}\in \ctmdaa$ and $\dt\du^h u^{I\! I}\in \cC_T(H^{m-2h-1}_{\ast\ast})$ so that from Lemma \ref{prodotto2} we have $\du A_0^{I\! I\,I\! I}\dt\du^h u^{I\! I}\in \cC_T(H^{m-2h-1}_{\ast\ast})$ because $m-2\ge4$. This allows to estimate the first term because $\gamma +2(k-h-1)\le m-2h-1$. We argue similarly for the second term in the right-hand side of \eqref{exp-commut-A0}. For the third term we argue as for \eqref{exp-comm-A0}. We obtain
\begin{equation}\label{gamma-k-A0}
\begin{array}{ll}
||[\l^{-\gamma_0}\dg_\star\du^k,A_0^{I\! I\,I\! I}\dt]u^{I\! I}||\leq
C||[A_0]||_{m,\ast\ast,\l}|||u|||_{m,\ast\ast,\l} ,\qquad
t\in[0,T],
\end{array}
\end{equation}
provided that $m\ge6$.
In a similar way we get
\begin{equation}\label{gamma-k-Aj}
\begin{array}{ll}
||[\l^{-\gamma_0}\dg_\star\du^k,A_j^{I\! I\,I\! I}\dj]u^{I\! I}||\leq
C|||A_j|||_{m,\ast\ast,\l}|||u|||_{m,\ast\ast,\l},\quad
t\in[0,T],
\quad j=2,3.
\end{array}
\end{equation}
Moreover, we expand the commutator with $A_1$ as
\begin{equation}
\begin{array}{ll}\label{exp-commut-A1}
\displaystyle   [\dg_\star\du^k,A_1^{I\! I\,I\! I}\du]u^{I\! I}=  \sum_{h=1}^{k-1}\dg_\star\duno^{k-h}( \du A_1^{I\! I\,I\! I}\du^hu^{I\! I}) + [\dg_\star,A_1^{I\! I\,I\! I}\duno]\du^ku^{I\! I}.
\end{array}
\end{equation}
For the first summation we argue as before, while for the commutator we argue as for \eqref{exp-commutA1}. We obtain
\begin{equation}\label{gamma-k-A1}
\begin{array}{ll}

||[\l^{-\gamma_0}\dg_\star\du^k,A_1\du]u||\leq
C|||A_1|||_{m,\ast\ast,\l}|||u|||_{m,\ast\ast,\l},\qquad
t\in[0,T],
\end{array}
\end{equation}
provided that $m\ge6$.

Now we estimate $G$. We have
\begin{equation}\label{estG1}
\begin{array}{ll}
\displaystyle
||\l^{-\gamma_0}\dg_\star\du^k G_1||=||\l^{-\gamma_0}\dg_\star\du^k{\l\dd q\choose\l\dtr
q}||=||\l^{-\gamma_0}\dg_\star\du^{k-1}{\dd\choose\dtr}\l\du q||\leq |||\l\du
q|||_{m-1,\ast\ast,\l}
\end{array}
\end{equation}
because $|\gamma|+2k\le m+1$ is equivalent to $|\gamma|+1+2(k-1)\le m$; moreover we can obtain
\begin{equation}\label{estG2}
\begin{array}{ll}
||\l^{-\gamma_0}\dg_\star\du^k
G_2||\leq
||\l^{-\gamma_0}\dg_\star\du^{k-1}\left\{{\rho_p\over\rho}{H\over\l}(\dt\du
q+v\cdot\nabla\du q)+\begin{pmatrix}
 H\cdot\nabla \du
v_1\\0\\0
 \end{pmatrix}\right\}||
\\
+||\l^{-\gamma_0}\dg_\star\du^{k-1}\left\{\du({\rho_p\over\rho}{H\over\l})\dt
q+\du({\rho_p\over\rho}{H\over\l}\,v)\cdot\nabla q+
\begin{pmatrix}
 \du
H\cdot\nabla
v_1\\0\\0\end{pmatrix}\right\}||\\
\leq
C|||A|||_{m,\ast\ast,\l}(|||u|||_{m,\ast\ast,\l}+|||\du
q|||_{m-1,\ast\ast,\l}+|||\du v_1|||_{m-1,\ast\ast,\l}),
\end{array}
\end{equation}
since
$$
|||{\rho_p\over\rho}{H\over\l}|||_{m,\ast\ast,\l}+
|||{\rho_p\over\rho}{H\over\l}\,v|||_{m,\ast\ast,\l}+|||H|||_{m,\ast\ast,\l}\leq
|||A|||_{m,\ast\ast,\l}. $$
From \eqref{gamma-k-A0}, \eqref{gamma-k-Aj}, \eqref{gamma-k-A1}--\eqref{estG2}, we obtain \eqref{tang-k-norm}. 
\end{proof}

\bigskip

%

Now we give a direct estimate of the normal derivatives of the \lq noncharacteristic\rq\ part of the solution $q_1,v_1,H_1$.

\begin{lemma}\label{3.6}
Let $u$ be a sufficiently smooth solution of \eqref{ibvp}
satisfying
\eqref{set}.  

(i) The pressure $q$ satisfies the estimate
\begin{equation}\label{stima-norm-q}
\begin{array}{ll}
|||\l\du q|||_{m-1,\ast\ast,\l}\leq \Phi(
K)\left( |||u|||_{m,\ast\ast,\l}+|||v_1,H_1|||_{m,\ast\ast\ast,\l}\right),
\end{array}
\end{equation}
for all $t\in[0,T],\,\l\geq1$, for a suitable function $\Phi$ independent of $\l$.

(ii) There exists $\l_0\ge1$ sufficiently large (dependent on $|||u|||_{5,\ast\ast,\l}$) such that
\begin{equation}\label{stima-norm-v1}
\begin{array}{ll}

|||\du v_1|||_{m-1,\ast\ast,\l}\leq \Phi(
K) |||u|||_{m,\ast\ast,\l},
\end{array}
\end{equation}
for all $t\in[0,T],\,\l\geq
\l_0$, for a suitable function $\Phi$ dependent on $\l_0$ but independent of $\l$.

(iii) The components of the solution $q_1,v_1,H_1$ satisfy
\begin{equation}\label{stima-3ast}
\begin{array}{ll}

|||q_1,v_1,H_1|||_{m,\ast\ast\ast,\l}\leq \Phi(
K) |||u|||_{m,\ast\ast,\l},
\end{array}
\end{equation}
for all $t\in[0,T],\,\l\geq
\l_0$,  for a suitable function $\Phi$ dependent on $\l_0$ but independent of $\l$.

\end{lemma}
\begin{proof}
We estimate the normal derivative of $q$ from \eqref{duno q} by using Lemmata \ref{nuovo3}, \ref{prodotto2}. Specifically, for the critical term $ v_1\du v_1$ we have
\begin{equation}\label{critic1}
\begin{array}{ll}
 |||v_1\du v_1|||_{m-1,\ast\ast,\l}
  &\displaystyle \le C\left(
\|\frac{ v_1}{\sigma}\|_{W^{1,\infty}_\ast(Q_T)}|||\sigma\du v_1|||_{m-1,\ast\ast,\l}+\|\sigma\du v_1\|_{W^{1,\infty}_\ast(Q_T)}|||\frac{ v_1}{\sigma}|||_{m-1,\ast\ast,\l}\right)\\
&\displaystyle  \le C\left(
|||\frac{ v_1}{\sigma}|||_{3,\ast\ast,\l}|||v_1|||_{m,\ast\ast,\l}+|||\sigma\du v_1|||_{3,\ast\ast,\l}|||\frac{ v_1}{\sigma}|||_{m-1,\ast\ast,\l}\right)\\
\displaystyle  &\le 
C|||v_1|||_{4,\ast\ast\ast,\l}|||v_1|||_{m,\ast\ast\ast,\l}\le 
C|||v_1|||_{5,\ast\ast,\l}|||v_1|||_{m,\ast\ast\ast,\l},
\end{array}
\end{equation}
where for the last inequalities we have used Theorem \ref{main3} and Lemma \ref{lemma-base}(i). A similar argument gives
\begin{equation}\label{critic2}
\begin{array}{ll}
\displaystyle |||H_1\du H_1|||_{m-1,\ast\ast,\l}  \le 
C|||H_1|||_{5,\ast\ast,\l}|||H_1|||_{m,\ast\ast\ast,\l}.
\end{array}
\end{equation}
The estimate of the other terms in \eqref{duno q} is straightforward and so we obtain \eqref{stima-norm-q}.
For the estimate of $\duno v_1$ we consider \eqref{mhd21} that gives
\begin{equation}\label{duno v_1}
\duno v_1=-\ddue v_2-\dtre v_3 -{\rho_p\over\l\rho}\{(\partial_t+v\cdot\nabla)q-{H\over\l}
\cdot(\partial_t+(v\cdot\nabla))H \}.
\end{equation}
The most critical terms in the right-hand side of \eqref{duno v_1} are $ v_1\du q$ and $ v_1\du H$, that are estimated by applying the arguments used for \eqref{critic1} to obtain
\begin{equation}\label{critic3}
\begin{array}{ll}
\displaystyle |||v_1\du q|||_{m-1,\ast\ast,\l}  &\displaystyle  \le C\left(
|||\frac{ v_1}{\sigma}|||_{3,\ast\ast,\l}|||q|||_{m,\ast\ast,\l}+|||\sigma\du q|||_{3,\ast\ast,\l}|||\frac{ v_1}{\sigma}|||_{m-1,\ast\ast,\l}\right)\\
&\displaystyle  \le C\left(|||v_1|||_{5,\ast\ast,\l}|||q|||_{m,\ast\ast,\l}+|||q|||_{4,\ast\ast,\l}|||v_1|||_{m,\ast\ast\ast,\l} \right),
\end{array}
\end{equation}
\begin{equation}\label{critic4}
\begin{array}{ll}

\displaystyle |||v_1\du H|||_{m-1,\ast\ast,\l}    \le C\left(
|||v_1|||_{5,\ast\ast,\l}|||H|||_{m,\ast\ast,\l}+|||H|||_{4,\ast\ast,\l}|||v_1|||_{m,\ast\ast\ast,\l} \right).
\end{array}
\end{equation}
After estimating the other terms in the right-hand side of \eqref{duno v_1} we obtain
\begin{equation}\label{est-duno v_1}
\begin{array}{ll}

\displaystyle |||\du v_1|||_{m-1,\ast\ast,\l}    \le 
|||u^{I\! I}|||_{m,\ast\ast,\l}+\frac{\Phi(K)}{\l} \left(|||u|||_{m,\ast\ast,\l}+|||v_1|||_{m,\ast\ast\ast,\l} \right).
\end{array}
\end{equation}
Considering that 
\[
|||v_1|||_{m,\ast\ast\ast,\l}\le |||v_1|||_{m,\ast\ast,\l}+|||\duno v_1|||_{m-1,\ast\ast,\l},
\]
see Lemma \ref{lemma-base}(ii), we can obtain from \eqref{est-duno v_1} the estimate of the form
\begin{equation}\label{est-duno v_1n2}
\begin{array}{ll}

\displaystyle |||\du v_1|||_{m-1,\ast\ast,\l}    \le 
{\Phi(K)} |||u|||_{m,\ast\ast,\l},
\end{array}
\end{equation}
for all $\l\ge\l_0$ sufficiently large and for a suitable function $\Phi$, that is \eqref{stima-norm-v1}. $\l_0$ depends on $|||u|||_{5,\ast\ast,\l}$.

\noindent
Finally, from the divergence constraint \eqref{divH} we readily obtain
\[
 |||\du H_1|||_{m-1,\ast\ast,\l}    \le 
 |||H_2|||_{m,\ast\ast,\l}+ |||H_3|||_{m,\ast\ast,\l}\le  2|||u|||_{m,\ast\ast,\l},
\]
and from Lemma \ref{lemma-base}(ii) we have 
\begin{equation}\label{est-H1}
\begin{array}{ll}

\displaystyle ||| H_1|||_{m,\ast\ast\ast,\l}    \le C
  |||u|||_{m,\ast\ast,\l}.
\end{array}
\end{equation}
Analogously, from \eqref{stima-norm-v1} we obtain
\begin{equation}\label{est-v1}
\begin{array}{ll}

\displaystyle ||| v_1|||_{m,\ast\ast\ast,\l}    \le 
{\Phi(K)} |||u|||_{m,\ast\ast,\l}.
\end{array}
\end{equation}
Then from \eqref{stima-norm-q}, \eqref{est-H1}, \eqref{est-v1} we have
\begin{equation}\label{est-q}
\begin{array}{ll}
\displaystyle ||| q|||_{m,\ast\ast\ast,\l}    \le 
{\Phi(K)} |||u|||_{m,\ast\ast,\l}.
\end{array}
\end{equation}
The three estimates \eqref{est-H1}--\eqref{est-q} give \eqref{stima-3ast}.
\end{proof}

\bigskip

Set for convenience
\begin{equation}\label{defM}
\begin{array}{ll}
M(t)=
\displaystyle
\Big(\sum_{|\a|\leq m}||\l^{-\a_0}\partial^\alpha_\star
u||^2+\sum_{|\b|\leq
m-1}||\l^{-\b_0}\partial^\b_\star\du u||^2
\\
\displaystyle\qquad\qquad\quad   +\sum_{k\geq 2}\;\sum_{|\gamma|+2k\leq
m+1}||\l^{-\gamma_0}\dg_\star\du^k u^{I\! I}||^2\Big)^{1/2}.
\end{array}
\end{equation}
From Lemmata \ref{stimeA}, \ref{lemma-stime-tang} -- \ref{3.6} we first obtain
\begin{equation}\label{stimaH1}
\begin{array}{ll}
\displaystyle a_0M^2(t)\leq a_0^{-1}M^2(0)+\int^t_0\Phi(K)\{M^2(s)+|||u(s)|||_{m,\ast\ast,\l}^2+|||u(s)|||_{m,\ast\ast,\l}^3\}\,ds
.
\end{array}
\end{equation}
Then we observe that by definition
\begin{equation}\label{3.58}
\begin{array}{ll}
|||u(t)|||_{m,\aa,\l}\leq M(t)+|||\du q(t),\du v_1(t)|||_{m-1,\ast,\l},
\end{array}
\end{equation}
and from \eqref{stima-norm-q}, \eqref{stima-3ast}, \eqref{est-duno v_1} we can obtain 
 \begin{equation*}
\begin{array}{ll}\label{}
\displaystyle  |||\du q(t),\du v_1(t)|||_{m-1,\ast,\l} \le  |||u^{I\! I}(t)|||_{m,\ast\ast,\l}+\frac{\Phi(K)}{\l} |||u(t)|||_{m,\ast\ast,\l}.
\end{array}
\end{equation*}
Substituting in \eqref{3.58} and observing that $ |||u^{I\! I}(t)|||_{m,\ast\ast,\l}\le M(t)$ we get
\begin{equation*}\label{}
\begin{array}{ll}
\displaystyle |||u(t)|||_{m,\aa,\l}\leq M(t)+\frac{\Phi(K)}{\l} |||u(t)|||_{m,\ast\ast,\l}.
\end{array}
\end{equation*}
For $\l$ sufficiently large, say for $\l\ge\l_0$ again, with $\l_0$ dependent on $K$, we obtain 
\begin{equation}\label{uH}
\begin{array}{ll}
\displaystyle |||u(t)|||_{m,\aa,\l}\leq {\Phi_1(K)}M(t),
\end{array}
\end{equation}
for every $t\in[0,T]$ and $\l\ge\l_0$, for a suitable increasing function $\Phi_1$. 
We substitute \eqref{uH} in \eqref{stimaH1} and obtain the following lemma:

\begin{lemma}\label{3.8}
Let $u$ be a sufficiently smooth solution
of \eqref{ibvp}
satisfying
\eqref{set}. Then there exists an increasing function $\Phi$ independent of $\l$ such that $M(t)$, defined in (3.51), obeys the estimate
\begin{equation}\label{stimaH2}
\begin{array}{ll}
\displaystyle M^2(t)\leq a_0^{-2}M^2(0)+\int^t_0\Phi(K)\{M^2(s)+M^3(s)\}\,ds
.
\end{array}
\end{equation}
for every $t\in[0,T]$ and $\l\ge\l_0$, with $\l_0$ dependent on $K$. 
\end{lemma}

\bigskip

\section{Proof of Theorem 2.2}

First of all we observe that (2.10) yields from (2.4) that $\l^{-1}u_0^{\l(1)}$ is
bounded in $H^{m-1}(  \O)$, uniformly in $\l$. By repeated differentiation in time we verify
that $\l^{-k}u_0^{\l(k)}$ is uniformly bounded in $H^{m-k}(  \O)$. It follows that
$$
M^\l(0)\leq|||u_0^\l|||_{m,\aa,\l}\leq C_1(K_1)\qquad\forall\l\geq1,
$$
where $M^\l(t)$ replaces $M(t)$, defined in \eqref{defM}, when $u=u^\l$. Since we have
$$
A_0(u^\l(t)/\l)=A_0(u^\l_0/\l)+\int^t_0{\partial A_0\over \partial
u}\l^{-1}\dt u^\l(s)ds
$$
and similar equations for the functions $\rho(u^\l(t)/\l)$ and
$(\rho_p/\rho)(u^\l(t)/\l)$, by means of the imbedding $C([0,T];H^{m-2}_\aa(  \O))\hookrightarrow
 C^0_B(\overline Q_T)$ we show
\begin{equation}\label{4.1}
\begin{array}{ll}
|A_0(u^\l(t)/\l)-A_0(u^\l_0/\l)|\leq \Phi_2(|||u^\l|||_{m,\aa,\l,T})T,\\
\\
|\rho(u^\l(t)/\l)-\rho(u^\l_0/\l)|
+|(\rho_p/\rho)(u^\l(t)/\l)-(\rho_p/\rho)(u^\l_0/\l)|
\leq
\Phi_2(|||u^\l|||_{m,\aa,\l,T})T
\end{array}
\end{equation}
for all $\l\geq1$ and $ t\in[0,T]$, and for a suitable increasing function $\Phi_2$. Let us choose $K$ in \eqref{set} as
\[
K=2C_1(K_1).
\]
We take $T>0$ such that any solution $M^\l(t)$ of \eqref{stimaH2} obeys
\begin{equation}\label{4.4}
\begin{array}{ll}

M^\l(t)\leq 2a_0^{-1}C_1(K_1)\qquad\forall t\in[0,T],
\end{array}
\end{equation}
and we also assume that $T$ is so small that
\begin{equation}\label{4.5}
\begin{array}{ll}

\Phi_2(2a_0^{-1}\Phi_1(K)C_1(K_1))T\leq\min\{a_0,(2a_0)^{-1},R/2\},\\
2a_0^{-1}\Phi_1(K)T\leq 1.
\end{array}
\end{equation}

It follows from \eqref{uH}, \eqref{stimaH2} and the Gronwall inequality, \eqref{4.1}--\eqref{4.5}, that  $u^\l$ satisfies the uniform estimates
\begin{equation}\label{4.7}
\begin{array}{ll}
|||u^\l|||_{m,\aa,\l,T}\leq
\Phi_1(K)\max_{t\in[0,T]}M^\l(t)\leq 2a_0^{-1}C_1(K_1)\Phi_1(K),
\end{array}
\end{equation}
\begin{equation}\label{4.71}
\begin{array}{ll}
|||u^\l|||_{5,\aa,\l,T}&\leq
|||u_0^\l|||_{5,\aa,\l}+T|||\dt u^\l|||_{m-1,\aa,\l,T}\\
&\leq C_1(K_1)+ 2a_0^{-1}C_1(K_1)\Phi_1(K)T \leq 2C_1(K_1)=K \qquad\forall\l\geq\l_0,
\end{array}
\end{equation}
and the other requirements of \eqref{set} on the interval
$[0,T]$.

Since $|||u^\l|||_{m,\aa,\l,T}$ depends boundedly on the parameter $\l$ (i.e. it is
bounded uniformly for $\l$ in the bounded intervals $0<\l'\leq\l\leq\l''$), the uniform estimate \eqref{4.7} holds
with a suitable constant for all $\l\geq1$. This gives the first part of \eqref{u-uniform}. 
Directly from
\eqref{mhd23} we obtain 
\begin{equation}\label{4.8}
\begin{array}{ll}
|||\dt H^\l|||_{m-1,\aa,\l,T}\leq C
|||u^\l|||_{m,\aa,\l,T}^2\leq C(K_1),
\end{array}
\end{equation}
which gives the second part of \eqref{u-uniform}. The proof of Theorem 2.2 is complete.

\section{Proof of Theorem 2.3}

From \eqref{u-uniform}, we immediately deduce the
existence of a subsequence, again denoted by $u^\l$, and functions $(q^\infty, w,
B)$ such that
\begin{equation}\label{5.1}
\begin{array}{ll}

u^\l\rightarrow(q^\infty,w,B)\quad\hbox{weakly-$\ast$ in } L^\infty(0,T;\hmaa)
\end{array}
\end{equation}
as
$\l\rightarrow+\infty$. Moreover, from the Ascoli-Arzel\`a theorem and Theorem \ref{th-compact} we also
deduce  $$H^\l\rightarrow B\inn C([0,T];H^{m-1}_{\ast,loc}(  \O)).$$
Let $S$ be the closure of $\{v\in C^\infty_0(\O);\nabla\cdot v=0\}$ in $L^2(  \O)$ and $G$ be the
orthogonal complement of $S$ in $L^2(  \O)$. Then we have $L^2(  \O)=S\oplus G$ by
the Helmholtz decomposition. We denote by $P_S,P_G$ the orthogonal projections in
$L^2(  \O)$ onto $S$ and $G$, respectively. It is well known that
$P_S,P_G\in\cL(H^m(  \O),H^m(  \O))$ for $m\geq 0$. The convergence of $v^\l$ is shown
by applying the result of \cite{MR1458527}, which uses the dispersion of the acoustic part
$(q^\l,P_Gv^\l)$ of the flow in the unbounded domain. We write \eqref{mhd21}, \eqref{mhd22}, in the form
(we use the notation of \cite{MR1458527}) 
\begin{equation}\label{5.2}
\begin{array}{ll}
\dt q^\l+\l\mu_1\nabla\cdot v^\l=\G^0,\cr
\dt v^\l+\l\mu_2\nabla q^\l=\G,
\end{array}
\end{equation}
where $\mu_1=({\overline\rho_p\over\overline\rho})^{-1},\mu_2=(\overline\rho)^{-1}$, $\overline\rho=\rho(0)$, 
$\overline\rho_p=\rho_p(0)$, 
\begin{equation*}
\begin{array}{ll}\label{}
\G^0=({\overline\rho_p\over\overline\rho})^{-1}\{
({\overline\rho_p\over\overline\rho}-{\rho_p^\l\over\rho^\l})\partial_tq^\l
-{\rho_p^\l\over\rho^\l}[v^\l\cdot\nabla q^\l-{H^\l\over\l}
\cdot(\partial_t+v^\l\cdot\nabla)H^\l]\},\cr

\G=(\overline\rho)^{-1}\{(\overline\rho-\rho^\l)\partial_tv^\l-\rho^\l(v^\l\cdot\nabla)v^\l+
(H^\l\cdot\nabla)H^\l\}.
\end{array}
\end{equation*}
Set $\F=(\G^0,\G)$. Then we show 
\medskip

\begin{lemma}\label{lemma5.1}
There exists a positive constant $C$ independent of $\l$ such that
\begin{equation}\label{5.3}
\begin{array}{ll}
|\F(t)|_1+||\F(t)||_3\leq C\qquad\forall\l\geq 1,t\in[0,T].
\end{array}
\end{equation}
Moreover, the first component $\G_1(t)\in H^1_0(  \O)$ for any $\l\geq 1$ and $t\in[0,T]$.
\end{lemma}

\begin{proof}

The estimate \eqref{5.3} is an easy consequence of the uniform estimate
\eqref{u-uniform}. We only observe that we have
$$
({\rho_p^\l\over\rho^\l}-{\overline\rho_p\over\overline\rho})\partial_tq^\l=
(q^\l-|H^\l|^2/( 2\l))\dt
q^\l/\l\int_0^1({\partial\,\over\partial
p}{\rho_p^\l\over\rho^\l})(\tau(q^\l/\l-|H^\l/\l|^2/2))d\tau$$
which shows that
$({\rho_p^\l/\rho^\l}-{\overline\rho_p/\overline\rho})\partial_tq^\l$ is
uniformly bounded in $C([0,T];H^{m-1}_\aa(  \O))$ $\hookrightarrow$ $ C([0,T];H^3(  \O))$ and in $
C([0,T];L^1(  \O))$.  $(\rho^\l-\overline\rho)\dt v^\l$ is treated similarly.
The second part of the lemma follows directly from the boundary conditions $v_1=H_1=0$
on $\Gamma$. 
\end{proof}

\medskip
\begin{lemma}\label{lemma5.2}
There exists a positive constant $C$ independent of $\l$ such that
\begin{equation}\label{5.4}
\begin{array}{ll}
||P_Sv^\l(t)||_3+||\dt P_Sv^\l(t)||_3\leq C\qquad\forall\l\geq 1,t\in[0,T].
\end{array}
\end{equation}
\end{lemma}
\begin{proof}
From \eqref{u-uniform}, $v^\l$ is uniformly bounded in
$C([0,T];H^m_\aa(  \O))\hookrightarrow$ $ C([0,T];H^3(  \O))$. The first part of \eqref{5.4} then
follows from $P_S\in\cL(H^3(  \O),H^3(  \O))$. The second part follows from $\dt
P_Sv^\l=P_S\G$ and \eqref{5.3}. 
\end{proof}

\medskip
The rest of the proof proceeds as in \cite[Lemma 4.3]{MR1458527} and the following arguments. 
Thus we show $q^\infty=0$ that, together with \eqref{5.1}, gives \eqref{conv-u} and we show the
convergence of $v^\l$ as in \eqref{conv-v}. Actually in \cite{MR1458527} system \eqref{5.2} has only one and the same
coefficient $\mu$ instead of $\mu_1,\mu_2$, as we have. However by the change in
scale $q^\l=\a r^\l$ where $\a=\overline\rho({\overline\rho_p})^{-1/2}$ we easily
reduce to that case. We obtain 
\begin{equation}\label{5.5}
\begin{array}{ll}
 w(t)=P_Sw_0-\int^t_0
P_S\{(w\cdot\nabla)w-(\overline\rho)^{-1}(B\cdot\nabla)B\}ds. 
\end{array}
\end{equation}
The
limit satisfies $w,B\in L^\infty(0,T;H^3(  \O))$ and consequently $w\in Lip(0,T;H^2( \O))$
from \eqref{5.5}. Then it follows 
\begin{equation*}
\begin{array}{ll}\label{}
\dt w(t)=-P_S\{(w\cdot\nabla)w-(\overline\rho)^{-1}(B\cdot\nabla)B\}, \qquad
w_{|t=0}=P_Sw_0,
\end{array}
\end{equation*}
which is the abstract form of \eqref{limit-equ}$_1$. The passage to the limit
in \eqref{mhd23} is easily obtained by using \eqref{conv-u}, \eqref{conv-H} and gives \eqref{limit-equ}$_2$. Since \eqref{limit-equ} has a
unique solution in the above class $L^\infty(0,T;H^3( \O))\cap
Lip(0,T;H^2(  \O))$ we have the convergence of the whole sequences. It follows from \cite{MR1257135} that
$w,B\in\ctm$. Moreover we have
$$
\l(\rho^\l-\overline\rho)=
(q^\l-|H^\l|^2/2\l)\int_0^1\rho_p(\tau(q^\l/\l-|H^\l/\l|^2/2))d\tau,
$$
which is bounded in $C([0,T];H^{m}_{\aa}(  \O))$, so that \eqref{conv-rho} readily follows.

\section{Proof of Theorem 2.4}

We observe that \eqref{well-prep} yields from \eqref{mhdq} that $u_0^{\l(1)}$ is
bounded in $H^{m-1}(  \O)$, uniformly in $\l$. By repeated differentiation in time we verify
that $\l^{-k+1}u_0^{\l(k)}$ is uniformly bounded in $H^{m-k}(  \O)$. It follows that
$$
|||u_0^{\l(1)}|||_{m-1,\aa,\l}\leq C\qquad\forall\l\geq1.
$$
Then we proceed as for \eqref{u-uniform}, the only difference being in the different dependence on
the weight $\l^{-1}$, and obtain 
\begin{equation}\label{u-ut-uniform}
\begin{array}{ll}
|||u^{\l}|||_{m,\aa,\l,T}+|||\dt u^{\l}|||_{m-1,\aa,\l,T}\leq
C\qquad\forall\l\geq1, 
\end{array}
\end{equation}
which in turn gives from \eqref{mhdq}
$$
\l|||\nabla q^{\l}|||_{m-1,\aa,\l,T}+\l|||\nabla\cdot v^{\l}|||_{m-1,\aa,\l,T}\leq
C\qquad\forall\l\geq1. $$
Thus \eqref{u-uniform2} is shown. It readily follows that 
$$
\nabla q^\l\rightarrow 0,\quad\nabla\cdot v^\l\rightarrow 0
\quad\hbox{in } C([0,T];H^{m-1}_{\aa}(  \O))
$$
and from \eqref{u-ut-uniform}, the Ascoli-Arzel\`a theorem and Theorem \ref{th-compact},
$$
q^\l\rightarrow 0,\quad v^\l\rightarrow w
\quad\hbox{in }C([0,T];H^{m-1}_{\ast,loc}(  \O)),
$$
that is \eqref{conv-qv}.
Now we show \eqref{conv-grad-q}. We take the time derivative of \eqref{5.2}
\begin{equation}\label{sist-vt}
\begin{array}{ll}
\dt q^\l_t+\l\mu_1\nabla\cdot v^\l_t=\G^0_t,\\
\dt v^\l_t+\l\mu_2\nabla q^\l_t=\G_t,\
\end{array}
\end{equation}
where $q^\l_t=\dt q^\l,v^\l_t=\dt v^\l$ and so on.
Set $\F_t=(\G^0_t,\G_t)$.
\begin{lemma}\label{6.1}
There exists a positive constant $C$ independent of $\l$ such that
\begin{equation*}
\begin{array}{ll}\label{}
|\F_t(t)|_1+||\F_t(t)||_2\leq C,\qquad
||P_Sv^\l_t(t)||_3+||\dt P_Sv^\l_t(t)||_2\leq C
\end{array}
\end{equation*}
for all $\l\geq 1,t\in[0,T]$. Moreover, the first component $\G_{t1}(t)\in H^1_0(  \O)$. 
\end{lemma}

\begin{proof}
The proof follows as in Lemmata \ref{lemma5.1} and \ref{lemma5.2} by observing that $\F_t$
is uniformly bounded in $C([0,T];H^{m-2}_\aa(  \O))$, $v^\l_t$ is uniformly bounded in
$C([0,T];H^{m-1}_\aa(  \O))$ and the imbeddings $H^{m-2}_\aa(  \O)\hookrightarrow H^2(  \O),H^{m-1}_\aa(  \O)\hookrightarrow H^3(  \O)$ hold.
\end{proof}

Now we repeat the argument of \cite{MR1458527} for the solution to \eqref{sist-vt} and obtain
\begin{equation}\label{conv-dtv}
\begin{array}{ll}
\dt v^\l\rightarrow \dt w\quad\hbox{in }
C_{loc}((0,T]\times\overline  \O).
\end{array}
\end{equation}
By \eqref{u-uniform}, \eqref{stima-3ast} and Lemma \ref{prodotto2}(ii) and Theorem \ref{main3}, we have that $(v^\l\cdot\nabla)v^\l$ is bounded uniformly in
$C([0,T];H^{m-1}_\aa(  \O))$ with $\dt (v^\l\cdot\nabla)v^\l$ bounded uniformly in
$C([0,T];H^{m-2}_\aa(  \O))$. Then the usual compactness argument gives the convergence of
$(v^\l\cdot\nabla)v^\l$ to $(w\cdot\nabla)w$ in $C([0,T];H^{m-2}_{\ast,loc}(  \O))$ which
yields
\begin{equation}\label{}
\begin{array}{ll}
(v^\l\cdot\nabla)v^\l\rightarrow (w\cdot\nabla)w\quad\hbox{in }
C_{loc}([0,T]\times\overline  \O).
\end{array}
\end{equation}
Analogously
\begin{equation}\label{conv-BgradB}
\begin{array}{ll}
(H^\l\cdot\nabla)H^\l\rightarrow (B\cdot\nabla)B\quad\hbox{in }
C_{loc}([0,T]\times\overline  \O).
\end{array}
\end{equation}
Then \eqref{conv-grad-q} follows from \eqref{limit-equ}$_1$, \eqref{mhd22}, \eqref{conv-rho}, \eqref{conv-dtv}--\eqref{conv-BgradB}.

\appendix
\section{Properties of anisotropic space}\label{anisotropicspaces}

Several properties of the function spaces $\hma,\;\hmaa,\; \hmaaa$ are
contained in the
following lemmas and theorems. These results provide basic tools for calculations in
such spaces,
in particular in $\hmaa$, that we use in the present paper. For the proofs see \cite{MR2604255,MR1346224,MR1401431,MR1779863,MR2967997,techreport}. If
not explicitly stated, $\Omega$ is either $\R^n_+$
or a bounded open subset of
$\R^n$ with
$C^{\infty}$ boundary $\Gamma$. Recall that $[{n/2}]$ and $[{(n+1)/ 2}]$ denote the integer part of ${n/ 2}$ and ${(n+1)/ 2}$, respectively.

\begin{lemma}\label{lemma-base}
Let $m\geq 1$. (i) $H^{m+1}_{\ast}(  \O)\hookrightarrow \hmaa$, $H^{m+1}_{\ast\ast}(  \O)\hookrightarrow \hmaaa$.

\noindent(ii) The following characterizations hold:
\[\hmaa=\{u\in\hma : \dnu u\in H^{m-1}_\ast(  \O)\},\]
\[\hmaaa=\{u\in\hmaa : \dnu
u\in H^{m-1}_{\ast\ast}(  \O)\},\]
and there exists a constant $C>0$ such that
\begin{equation*}
\begin{array}{ll}\label{}
||u||_{H^m_{\ast\ast}(\Omega)}\le C(||u||_{H^m_{\ast}(\Omega)}+||\dnu u||_{H^{m-1}_{\ast}(\Omega)}) \quad \forall u\in H^m_{\ast\ast}(\Omega),
\end{array}
\end{equation*}
\begin{equation*}
\begin{array}{ll}\label{}

||u||_{H^m_{\ast\ast\ast}(\Omega)}\le C(||u||_{H^m_{\ast\ast}(\Omega)}+||\dnu u||_{H^{m-1}_{\ast\ast}(\Omega)}) \quad \forall u\in H^m_{\ast\ast\ast}(\Omega).
\end{array}
\end{equation*}

\end{lemma}

\medskip

%

\noindent
Let us define the space
$$
W^{1,\infty}_\ast(\Omega)=\{u\in L^\infty(\Omega)\, : \, Z_iu\in L^\infty(\Omega),\, i=1\dots,n\},
$$
equipped with its natural norm. We have the following Moser-type inequalities.
\begin{lemma}[\cite{MR2604255}]\label{nuovo3}
Let $m\in\N,\, m\ge1$. (i) For all functions $u$ and $v$ in $H^m_\ast(\Omega)\cap W^{1,\infty}_\ast(\Omega)$ one has $uv \in H^m_\ast(\Omega)$ and
\begin{equation}
\begin{array}{ll}\label{Moser-ast}
\|uv\|_{H^m_\ast(\Omega)}\le C(\|u\|_{H^m_\ast(\Omega)}\|v\|_{W^{1,\infty}_\ast(\Omega)}+\|u\|_{W^{1,\infty}_\ast(\Omega)}\|v\|_{H^m_\ast(\Omega)}).
\end{array}
\end{equation}
 (ii) For all functions $u$ and $v$ in $H^m_{\ast\ast}(\Omega)\cap W^{1,\infty}_\ast(\Omega)$ one has $uv\in H^m_{\ast\ast}(\Omega)$ and
\begin{equation}
\begin{array}{ll}\label{Moser-ast-ast}
\|uv\|_{H^m_{\ast\ast}(\Omega)}\le C(\|u\|_{H^m_{\ast\ast}(\Omega)}\|v\|_{W^{1,\infty}_\ast(\Omega)}+\|u\|_{W^{1,\infty}_\ast(\Omega)}\|v\|_{H^m_{\ast\ast}(\Omega)}).
\end{array}
\end{equation}
(iii) Assume $F=F(u)\in C^m$ with $F(0)=0$. If $u\in H^m_{\ast\ast}(\Omega)$ is such that $$\|u\|_{W^{1,\infty}_\ast(\Omega)}\le K,$$ then $F(u)\in H^m_{\ast\ast}(\Omega)$ and
\begin{equation}\label{}
\begin{array}{ll}
\|F(u)\|_{H^m_{\ast\ast}(\Omega)}\le C(K)\|u\|_{H^m_{\ast\ast}(\Omega)}.
\end{array}
\end{equation}

\end{lemma}

Imbedding theorems for the anisotropic spaces $H^m_\ast(\Omega)$ follow in natural way from the inclusion $H^m_\ast(\Omega)\hookrightarrow H^{[m/2]}(\Omega)$ and the imbedding theorems for standard Sobolev spaces, see \cite{MR1346224,MR1401431,MR1779863}. In particular, following this way one has the continuous imbedding $H^m_\ast(\Omega)\hookrightarrow C^0(\overline\Omega)$ if $m$ is such that $[m/2]>n/2$. This result is improved by the following theorem.
\begin{theorem}[\cite{MR2604255}]\label{imbedding}
Let $\Omega$ be a bounded open subset of $\R^n,\; n\geq 2$, with
$C^\infty$ boundary or the half-space $\Omega=\R^n_+$. Then the continuous imbedding $H^{[(n+1)/2]+1}_\ast(\Omega)\hookrightarrow C^0_B(\overline\Omega)$ holds. If $m>n/2$ then $H^{m}_{\ast\ast}(\Omega)\hookrightarrow C^0_B(\overline\Omega)$.
\end{theorem}
\medskip 

The next theorems deal with the product of two anisotropic Sobolev functions, one of which may have low order of smoothness. 
\begin{lemma}\label{prodotto}
Let $m\ge 1$ be an integer and $r=\max\left\{m,\left[\frac{n+1}{2}\right]+2\right\}$. If  $u\in\ctma$ and $v\in\ctra$ then $uv\in\ctma$ and
\begin{equation}\label{prodottogen}
|||u(t)v(t)|||_{m,\ast}\le C|||u(t)|||_{m,\ast}|||v(t)|||_{r,\ast}\,,\qquad t\in[0,T].
\end{equation}
\end{lemma}
\begin{lemma}\label{prodotto2}
Let $m\ge 1$ be an integer and $r=\max\left\{m,2\left[\frac{n}{2}\right]+2\right\}$. (i) If $u\in\ctma$ and $v\in\ctraa$ then $uv\in \ctma$ and
\begin{equation}\label{prodottogen3}
|||u(t)v(t)|||_{m,\ast}\le C|||u(t)|||_{m,\ast}|||v(t)|||_{r,\ast\ast}\,,\qquad t\in[0,T].
\end{equation}
(ii) If $u\in \ctmaa$ and $v\in \ctraa$ then $uv\in \ctmaa$ and
\begin{equation}\label{prodottogen2}
|||u(t)v(t)|||_{m,\ast\ast}\le C|||u(t)|||_{m,\ast\ast}|||v(t)|||_{r,\ast\ast}\,,\qquad t\in[0,T].
\end{equation}
\end{lemma}
\medskip 

The next theorems are important for the estimate of conormal derivatives.
\begin{theorem}[\cite{MR2967997}]\label{main}
Let $m\geq 2$.
Let $u\in H^m_\ast({\R}^n_+)\cap H^1_0({\R}^n_+)$ be a function and let $U$ be defined by
\begin{equation}
\begin{array}{ll}\label{defH}
U(x_1,x')
=u(x_1,x')/\sigma(x_1).
\end{array}
\end{equation}
Then
\begin{equation}
\begin{array}{ll}\label{hoH}
\displaystyle   \|U\|_{H^{m-2}_{\ast}({\R}^n_+)}\leq C\|u\|_{H^m_{\ast}({\R}^n_+)}.$$
\end{array}
\end{equation}
\end{theorem}
In the second anisotropic space $H^m_{\ast\ast}(\Omega)$ we have the following results. 
\begin{theorem}[\cite{MR2967997,techreport}]\label{main2}
Let $u\in H^m_{\ast\ast}({\R}^n_+)\cap H^1_0({\R}^n_+)$ for $m\geq 1$, and let $U$ be the function defined in \eqref{defH}. 
\begin{itemize}
\item[a.] If $m\geq2$  then
\begin{equation}
\begin{array}{ll}\label{hoH2}
\|U\|_{H^{m-1}_{\ast}({\R}^n_+)}\leq C\|u\|_{H^m_{\ast\ast}({\R}^n_+)}.
\end{array}
\end{equation}
\item[b.] If $m\ge3$ then
\begin{equation}
\begin{array}{ll}\label{hoH3}
\|U\|_{H^{m-2}_{\ast\ast}({\R}^n_+)}\leq C\|u\|_{H^m_{\ast\ast}({\R}^n_+)}.$$
\end{array}
\end{equation}
\end{itemize}
\end{theorem}

In the third anisotropic space $H^m_{\ast\ast\ast}(\Omega)$ we have the similar result. 
\begin{theorem}[\cite{techreport}]\label{main3}
Let $u\in H^m_{\ast\ast\ast}({\R}^n_+)\cap H^1_0({\R}^n_+)$ for $m\geq 2$, and let $U$ be the function defined in \eqref{defH}. 
Then
\begin{equation}
\begin{array}{ll}\label{hoH4}
\|U\|_{H^{m-1}_{\ast\ast}({\R}^n_+)}\leq C\|u\|_{H^{m}_{\ast\ast\ast}({\R}^n_+)}.
\end{array}
\end{equation}
\end{theorem}

The {last} theorem provides a compactness result.
\begin{theorem}[\cite{MR1779863}]\label{th-compact}
Let $\Omega$ be a bounded
open subset of
$\R^n,\; n\geq 2$, with $C^\infty$ boundary. Let
$m\geq 1$ be an integer. Then the imbedding $\hmaa\hookrightarrow\hmua$ is compact.

\end{theorem}

\section{Acknowledgement}
The research was supported in part by the Italian research
project PRIN 20204NT8W4-002.

\bibliographystyle{plain}

\bibliography{zero_Mach}

\end{document}